\documentclass[11pt,a4paper]{amsart}

\usepackage{amsmath,amssymb,mathrsfs,amsthm}
\usepackage{array,enumerate,verbatim}
\usepackage{WataruKai}
\usepackage[margin=0.8in]{geometry}
\usepackage{mathtools}\mathtoolsset{showonlyrefs=true} 
\usepackage{color}
\usepackage{booktabs} 

\usepackage{librebaskerville}

\usepackage[expansion=false]{microtype}

\usepackage{bm}
\usepackage[all]{xy}

\usepackage{hyperref}

\title[Green-Tao theorem in positive characteristic]{The Green-Tao theorem for affine curves over $\mathbb F _q$}
\author{Wataru Kai}
\address{Mathematical Institute,
Tohoku University,
6-3 Aoba, 980-8578 Sendai, Japan}
\email{kaiw@tohoku.ac.jp}

\begin{document}

\begin{abstract}
    Green and Tao famously proved in a 2008 paper that there are arithmetic progressions of prime numbers of arbitrary lengths.
    Soon after, analogous statements were proved by Tao for the ring of Gaussian integers
    and by L{\^e} for the polynomial rings over finite fields.
    In 2020 this was extented to orders of arbitrary number fields by Kai-Mimura-Munemasa-Seki-Yoshino.
    We settle the case of the coordinate rings of affine curves over finite fields.
    The main contribution of this paper is subtle choice of a polynomial subring of the given ring
    which plays the role of $\mathbb Z$ in the number field case.
    This choice and the proof of its pleasant properties eventually depend on the Riemann-Roch formula.

\end{abstract}

\maketitle

\tableofcontents

\section{Introduction}

In this paper we prove the following:

\begin{theorem}\label{thm:intro}
    Let $p$ be a prime.
    Let $\Or $ be an integral domain finitely generated over $\Fp $ and whose fraction field has transcendence degree $1$ over $\Fp $.
    Then for any positive integer $k\ge 1$, the set of prime elements of $\Or $ contains
    a $k$-dimensional affine subset.
\end{theorem}

Recall that an {\em affine subset} of a vector space is by definition a translate of a vector subspace (necessarily unique).
Its {\em dimension} is defined to be that of the corresponding vector subspace.

We actually prove a density version of the theorem which we now formulate.
Let $\Ded $ be the integral closure of $\Or $ in its fraction field.
It is a Dedekind domain finite over $\Or $.
There is a canonical linear norm, defined in \S \ref{sec:algebraic},
\begin{equation}
    \lnorm - \colon \Ded \to \bbR _{\ge 0}
,\end{equation}
which gives an increasing exhaustive filtration by finite subsets
\[ \Ded _{\le N} :=\br{\alpha \in \Ded \mid \lnorm \alpha \le N } \quad (N\ge 0). \]

For an inclusion $A\subset X\neq \varnothing $ of subsets of $\Ded $,
one can consider the {\em upper relative density}: 
\begin{equation}
    \bar\delta _X(A):= \limsup _{N\to +\infty } \frac{\mgn{A\cap \Ded _{\le N}} }{\mgn{ X\cap \Ded _{\le N}} }
.\end{equation}
The density statement is formulated as follows.
\begin{theorem}[{Green-Tao theorem in positive characteristic; see Theorem \ref{thm:homothetic}}]\label{thm:density} 
    Let $\Ded $ be a Dedekind domain finitely generated over $\Fp $ and $\mcal P_{\Ded }$ be the set of its prime elements.
    Then every subset $A\subset \mcal P_{\Ded }$ with $\bar\delta _{\mcal P_{\Ded }}(A)>0$
    contains a $k$-dimensional affine subset for an arbitrary $k\ge 0$.
\end{theorem}
This implies Theorem \ref{thm:intro} because the prime elements of $\Or $ have positive upper density in
$\mcal P_{\Ded }$ as we recall in \S \ref{sec:non-normal-case}.

In fact, in our proof of Theorem \ref{thm:density},
we search for $k$-dimensional affine subsets of a very specific form.
\begin{definition}
    For a subring $\ded \subset \Ded $ and a finite subset $S\subset \Ded $,
    by an {\em $\ded $-homothetic copy of $S$} let us mean a subset of $\Ded $ of the form
    \begin{equation}
        a\cdot S +\beta
        =\br{ a\alpha +\beta \mid \alpha \in S }
    \end{equation}
    with $a\in \ded $ and $\beta \in \Ded $.
    Let us say it is {\em non-trivial} if we can take $a\neq 0$.

    For later use, note that these notions make perfect sense for any integral domain $\ded $, a torsion-free $\ded $-module $\ideala $ and any subset $S\subset \ideala $.
\end{definition}
For a suitable $\ded \subset \Ded $ and an arbitrary $S$, we shall prove in Theorem \ref{thm:homothetic} that any subset of $\mcal P_{\Ded }$ with positive upper relative density contains a non-trivial $\ded $-homothetic copy of $S$.
Theorem \ref{thm:density} then follows because if we take $S$ to be a $k$-dimensional linear subspace of $\Ded $ then every non-trivial $\ded $-homothetic copy of it is a $k$-dimensional affine subset.
This is why we propose to call Theorem \ref{thm:density} the {\em Green-Tao theorem in positive characteristic}, as the Green-Tao theorem for number fields is commonly formulated as follows.

\begin{theorem}[{Green-Tao theorem for number fields: \cite{Green-Tao08}, \cite{Tao}, \cite{KMMSY}}]
    \label{thm:Green-Tao}
    Let $K$ be a number field and $\mcal O_K$ the ring of its integers.
    Denote by $\mcal P _K$ the set of prime elements of $\mcal O_K$.
    Then any set $A\subset \mcal P_K$
    with positive upper relative density
    contains a non-trivial $\bbZ $-homothetic copy of $S$ for an arbitrary finite subset $S$ of $\mcal O_K$.
\end{theorem}
In \cite{KMMSY} they also prove a variant of this statement for the ``prime elements'' (in an appropriate sense) in a given non-zero ideal $\ideala \subset \mathcal O_K$.
While we will also state and prove Theorem \ref{thm:homothetic} in this generality,
the reader is advised to assume $\ideala =\Ded $ in the first reading.
We will stick to the case $\ideala =\Ded $ in the rest of Introduction.

\subsection{Overview of the proof}

Theorem \ref{thm:density} for polynomial rings $\Ded  = \Fq [t]$ is due to L{\^e} \cite{Le}.
For Theorem \ref{thm:Green-Tao}, the case of $\bbZ $ is the renowned theorem of Green and Tao \cite{Green-Tao08}
and the case of $\bbZ [\kyo ]$ is due to Tao \cite{Tao}.
For the ring $\mcal O_K$ of integers in a general number field $K$, it is a result of
Mimura, Munemasa, Seki, Yoshino and the present author \cite{KMMSY}.
See Table \ref{table:Green-Tao-theorems}.
All of them follow the strategy of Green-Tao \cite{Green-Tao08}.

\begin{table}
\begin{tabular}{c@{\hspace{20pt}}|c@{\hspace{20pt}}l}
    &char. $0$& char. $p$
    \\[10pt] \hline
    the simplest case &$\bbZ$ (Green-Tao)    & $\Fq [t]$ (L{\^e})
    \\[10pt] 
    \te{more general cases} &
    \begin{tabular}{cl}$\bbZ [\kyo ]$ &  (Tao)  \\[1mm]%
    $\mcal O_K$&  \cite{KMMSY} \end{tabular}
    &$\Ded /\Fq$ (this article)
    \\ \hline
\end{tabular}

\caption{Green-Tao theorems for different rings}
\label{table:Green-Tao-theorems}
        \end{table}

        Details of our arguments are closest to those of \cite{KMMSY}.
The new issue we have to face is that while the number ring $\mcal O_K$ has a canonical base ring $\bbZ $
which is simple enough and
such that $\mcal O_K \cong \bbZ ^n$ compatibly with the metrics on both sides,
there is no canonical one for $\Ded $ in positive characteristic.
\vskip 0.5\baselineskip

In \S \ref{sec:algebraic} we use the {\em Riemann-Roch formula} to find an appropriate subring $\ded $ of $\Ded $,
which is isomorphic to the polynomial ring
and over which $\Ded $ is finite.
The subtlety of our choice is that, moreover, the $\ded $-linear isomorphism $\Ded \cong \ded ^\therank $ (given once we choose a basis)
is compatible with the metrics on both sides; see Proposition \ref{prop:equivalence-of-norms}.
Once we have done this, 
everything in \cite{KMMSY} goes through.
So we refer the reader to \cite[\S\S 1--2]{KMMSY} for a detailed overview. 

Let us just recall the three main ingredients:
\begin{itemize}
    \item the relative \Sz\ theorem (recalled in \S \ref{sec:Sz}); 
    \item the construction of a pseudorandom measure $\lambda \colon \Ded \to \bbR _{\ge 0}$ (\S \ref{sec:GY}, completed in \S \ref{sec:end-of-proof});
    \item 
    the prime elements $\mcal P_{\Ded }$ have positive density with respect to the measure $\lambda $ (\S \ref{sec:end-of-proof}).
\end{itemize}
Here, the relative \Sz\ theorem (Theorem \ref{thm:Sz}) roughly asserts the following:
suppose we are given a function $\lambda \colon \Ded \to \bbR _{\ge 0}$
which is a {\it pseudorandom measure}---this condition
says $\lambda $ is close enough to the constant function $1$ in a certain measure (see Definition \ref{def:S-N-rho-o-pseudorandom}).
Suppose also that a subset $A\subset \Ded $ has {\em positive (upper) density with respect to $\lambda $} in that
\begin{equation}\label{eq:what-it-means-to-be-dense}
    \limsup _{N\to +\infty }
    \frac{\sum\limits _{\alpha \in A\cap \Ded _{\le N}} \lambda (\alpha ) }{\mgn{\Ded _{\le N}}}
    >0
.\end{equation}
Then $A$ contains a non-trivial $\ded $-homothetic copy of $S$ 
for every finite subset $S\subset \Ded $.
\vskip 0.5\baselineskip
The construction of $\lambda \colon \Ded \to \bbR _{\ge 0}$ is ideal-theoretic in nature. The proof that it is pseudorandom ultimately relies on the knowledge that the zeta function $\zeta _{\Ded }$ has a simple pole at $1$ with a positive residue.

That the prime elements have positive density with respect to $\lambda $ will be deduced from the \Cheb\ density theorem, an analog of Prime Number Theorem in our setting.
\vskip 0.5\baselineskip

A pitfall in the construction of $\lambda $ is that there is a somewhat natural fuction $\lambda '\colon \Ded \to \bbR _{\ge 0}$ but
we are not going to use it directly
because we do not know if it is pseudorandom in our sense.
Instead, we choose an element $W\in \ded $ with sufficiently many different prime factors
and $b\in \Ded $ coprime to $W$,
and define the function $\lambda $ as the composite $\Ded \xrightarrow{W(-)+b} \Ded \xrightarrow{\lambda '}\bbR _{\ge 0}$ times a normalizing factor.
This enables us to prove the pseudorandomness.
Thus, the set $A\subset \Ded $ in the relative \Sz\ theorem is going to be taken as the inverse image of $\mcal P_{\Ded }$ by the affine linear map $W(-)+b$;
this is of course equivalent to considering the prime elements which are congruent to $b$ modulo $W$.
The reader will see how this trick (so-called {\em $W$-trick}) works as the proof unrolls in \S\S \ref{sec:GY}--\ref{sec:end-of-proof}.
\vskip 0.5\baselineskip

While the translation of the arguments in the number field case \cite{KMMSY} into our situation is straightforward in many places,
one cannot formally apply the results in \loccit\
because of the slight difference of languages
between number fields and algebraic curves.
So we include full proofs for the convenience of the reader.
Some details get simpler---as the reader might naturally expect---partly thanks to the fact that the canonical norm $\lnorm - \colon \Ded \to \bbR _{\ge 0}$ is ultrametric,
meaning that $\lnorm{\alpha +\beta }\le \max\br{\lnorm \alpha ,\lnorm \beta }$ for all $\alpha , \beta \in \Ded $,
so that the subsets $\Ded _{\le N}$ are in fact {\it subgroups}.

\subsection*{Notation}
For a function on a non-empty finite set $f\colon X\to \bbC $, we use the standard expectation notation:
\begin{equation}
    \bbE (f\emid X )=\bbE (f(x)\emid x\in X):=
    \frac 1{\mgn X} \sum _{x\in X} f(x).
\end{equation}

From \S \ref{sec:GY} onwards, we will make extensive use of the big-$O$ notation with dependence parameters
as in \cite[Notation in \S 2]{KMMSY}:
let $f$ and $g$ be $\mathbb C$-valued functions on a set $X$ which depend on additional parameters $a,b, c,\dots $.
Assume that the values of $g$ are positive real numbers.
We write
\begin{equation}
f(x)=O_{b,c,\dots } (g(x) ) \quad (x\in X)
\end{equation}
to mean that there is a positive constant $C=C_{b,c,\dots }>0$ depending only on the parameters in the subscript
such that the inequality $|f(x)| < C g(x)$ holds for all $x\in X$.
(Thus in this case the implied constant $C$ can be taken independent of $a$.)
An expression like $O_{b,c,\dots }(1)$ would mean a positive constant depending only on the parameters $b,c,\dots $ in the subscript.
Note in particular that $O(1)$ without subscript would mean an absolute constant depending on nothing at all.
When $g(x)$ is a heavy formula, the notation $f(x)=O_{b,c,\dots }(1)\cdot g(x)$ is sometimes preferred.
A quantity written in the form $1+O_{b,c,\dots }(g(x))$ is one whose difference with $1$ is $O_{b,c,\dots }(g(x))$.
When we say something like ``$f(x,y)=O_{b,c,\dots }(g(x,y) )$ for all sufficiently large $x,y\in \mathbb R$''
we are taking the domain $X$ to be a subset of $\mathbb R ^2$ consisting of the pairs of real numbers larger than certain thresholds.
In practice, the thresholds often depend on the parameters $a,b,c, \dots $. 
We usually indicate how the thresholds depend on the parameters, especially when that piece of information is relevant.


\section{Choice of a polynomial subring}
\label{sec:algebraic}

The purpose of this section is to define the canonical norm $\lnorm - \colon \Ded \to \bbR _{\ge 0}$,
set up necessary algebraic terminology
and find an appropriate subring $\ded \cong \Fq [t]$ of $\Ded $.

We use \cite{Rosen} as the main reference about algebraic background.
We assume the reader is familiar with
advanced undergraduate commutative algebra as in \cite{AM}
and basic notions of algebraic curves (equivalently function fields in one variable)
e.g.\ as in \cite[Chapter 5]{Rosen} \cite[Chapter I, Section 6]{Har}
but they do not have to know more than the Riemann-Roch theorem \cite[Theorem 5.4, p.49]{Rosen} \cite[Theorem 1.3 in Chapter IV, p.295]{Har}.

\subsection{The canonical linear norm}
Let $\Ded $ be a {\em Dedekind domain} finitely generated over $\Fp $.
The purpose of this subsection is to describe the canonical submultiplicative linear norm
on the ring $\Ded $.

By a {\em linear norm} or an {\em ultrametric norm} on an abelian group $G$ let us mean a non-negatively valued function $\lnorm - \colon G\to \bbR _{\ge 0}$
which is {\em ultrametric}:
\begin{equation}
    \lnorm {a+b} \le \max \br{ \lnorm a ,\, \lnorm b} \quad \te{ for all }a,b\in G
\end{equation}
and {\em non-degenerate} in that the only element with norm $0$ is the zero element.
A {\em submultiplicative} linear norm on an (always commutative) ring $A$ is a linear norm on the abelian group $A$ which moreover satisfies
\begin{equation}
    \lnorm {\alpha \beta } \le \lnorm \alpha \cdot \lnorm \beta .
\end{equation}
In our examples the multiplicative unit will always have norm $1$.
A norm $\lnorm -$ is said to be {\em multiplicative} if the above inequality is always an equality.

Let $\Fq $ be the integral closure of $\Fp $ in $\Ded $.
Let $\kurve $ be the complete non-singular curve over $\Fq $ which contains
$X:= \Spec \Ded$ as an open subscheme.
Let $n$ be the cardinality of the complement:
\begin{equation}
    n:= \mgn {\complm  }  .
\end{equation}
Regard each $v\in \kurve $ as a discrete valuation $\Frac (\Ded )^*\surj \bbZ $.
Let $\F (v)$ be its residue field and
$\deg (v):= [\F (v):\Fq ]$
its degree.
Define a linear norm $\Vert - \Vert _v $ on $\Ded $ by:
\begin{equation}
    \lnorm - _v \colon  \Ded \too \bbR ;\quad \alpha \mapsto \left( \frac{1}{ \mgn{\F (v)}}\right) ^{v(\alpha )}
    = q^{-v(\alpha )\cdot \deg (v)}
.\end{equation}
The value $\lnorm 0 _v$ is understood to be $0$.
The {\em canonical norm} $\lnorm{-} $ on $\Ded $ is defined by
\begin{equation}\label{eq:def-of-canonical-norm}
    \lnorm \alpha := \max _{v\in \complm } \{ \lnorm{\alpha } _v \} .
\end{equation}
This is a submultiplicative linear norm because the following formulas hold
for all $v\in \kurve $ and $\alpha ,\beta \in \Frac(\Ded )^*$:
\begin{equation}
    v(\alpha +\beta ) \ge \min \br{v(\alpha ),v(\beta )},
    \quad
    v(\alpha \beta ) = v(\alpha )+ v(\beta )
.\end{equation}

Define the {\em norm} of a non-zero ideal $\ideala \subset \Ded $ as the cardinality of the quotient
    $\Nrm (\ideala ):= \mgn{\Ded /\ideala } $,
and for $\alpha \in \Ded \nonzero $ write $\Nrm (\alpha ):= \Nrm (\alpha \Ded )$ for short.
By convention we define $\Nrm (0):=0$.
We call it the {\em ideal norm} of $\alpha $ to avoid confusion with the linear norm $\lnorm -$.
For $\alpha \neq 0$ we know
    $\Nrm (\alpha )= \prod _{v\in X}  \mgn{\bbF (v)}^{v(\alpha )} $
say by prime decomposition of ideals in $\Ded $.
It follows by the product formula for complete algebraic curves (e.g.\ \cite[Proposition 5.1, p.47]{Rosen}) that the following equality holds for all $\alpha \in \Ded $:
\begin{equation}\label{eq:product-formula}
    \Nrm (\alpha ) = \prod _{v\in \complm } \lnrom \alpha _v .
\end{equation}
In particular $\alpha $ is in $\Ded ^*$ if and only if $\prod _{v\in \complm } \lnrom \alpha _v =1$.

For positive real numbers $N>0$, set:
\begin{align}
    \Ded _{\le N} := &\br{\alpha \in \Ded \mid \lnorm \alpha \le N}
    \\
    =& \br{ \alpha \in \Ded \mid v(\alpha )+ \frac{\log _q N}{\deg (v)} \ge 0 \te{ for all }v\in \complm  }
    .
\end{align}
By the Riemann-Roch theorem for curves \cite[Corollary 4 of Theorem 5.4, p.49]{Rosen},
we know that $\mgn{\Ded _{\le N}}$ is approximately proportional to $N^n$.
To be more precise, consider the following invariants:
\begin{equation}
    \begin{array}{cl}
    g&:=\te{ the genus of $\kurve $},\\
    d_0&:= \te{ the least common multiple of $\deg (v)$'s for $v\in \complm $}
.\end{array}
\end{equation}
Let us denote by $\lfloor - \rfloor $ the floor function
$x\mapsto $ (the largest integer not exceeding $x$)
and
consider the next divisor on $\kurve $ for $N\ge 1$:
\begin{equation}
    D_N:= \sum _{v\in\complm} \left\lfloor \frac{\log _q N}{\deg (v)}\right\rfloor v .
\end{equation}
Its degree is $\sum _{v\in\complm } \left\lfloor \frac{\log _q N}{\deg (v)}\right\rfloor \deg (v)$
which is $\le \log _q N \cdot \mgn\complm = n\log _q N$.
We have by definition $\Ded _{\le N}=\Gamma (\kurve ,\mcal O_{\kurve }(D_N) )$,
or $\Ded _{\le N}=L(D_N)$ in the notation of \cite{Rosen}.
Therefore by Riemann-Roch \cite[Corollary 4 of Theorem 5.4, p.49]{Rosen}
we get $\mgn{\Ded _{\le N} }= N^n / q^{g-1}$
for every $N\ge q^{(2g-1)/n}$ which is a power of $q^{d_0}$.
From this we also get the following 
bound valid for all real numbers $N\ge q^{(2g-1)/n}$:
\begin{equation}\label{eq:size-of-Ded-N}
    \paren{\frac{N}{q^{d_0}}}^n/q^{g-1}
    < \mgn{\Ded _{\le N}}
    \le N^{n} /q^{g-1}
.\end{equation}

When we consider a non-zero ideal $\ideala \subset \Ded $
(which is relevant only if the reader is interested in the case $\ideala \neq \Ded $ of Theorem \ref{thm:homothetic}),
we endow $\ideala $ the induced linear norm and write
\begin{align}
    \ideala _{\le N} := &\br{\alpha \in \ideala \mid \lnorm \alpha \le N}
    \\
    =& \br{ \alpha \in \ideala \mid v(\alpha )+ \frac{\log _q N}{\deg (v)} \ge 0 \te{ for all }v\in \complm  }
    .
\end{align}
By the Riemann-Roch formula again (or from \eqref{eq:size-of-Ded-N}) we get
\begin{equation}\label{eq:size-of-ideala-N}
    \frac{N^n}{\Nrm (\ideala ) q^{nd_0+g-1}}
    < \mgn{\ideala _{\le N}}
    \le
    \frac{N^n}{\Nrm (\ideala ) q^{g-1}}
    \quad \te{ if }N^n \ge \Nrm (\ideala ) q^{2g-1}
.\end{equation}

\begin{remark}
    The use of the canonical norm among other norms is not essential.
    We could have chosen an arbitrary positive intger $d_v \ge 1$ for each $v\in\complm $ and defined
    $\lnorm \alpha _v := q^{-d_v\cdot v(\alpha )}$.
    The content of this paper would remain valid with minor modifications.
    However, it did not seem appealing to the author to allow the freedom of this choice
    at the cost of heavier notation.
\end{remark}

\begin{remark}
    The $A=\mcal P_{\Ded }$ case of Theorem \ref{thm:density} can be reduced to the case where $n=\mgn{\complm }=1$ (with a general $A\subset \mcal P_{\Ded }$),
    for which the treatment in the rest of \S \ref{sec:algebraic} can be much simpler
    because then we have $\lnorm - _{\Ded } = \lnorm - _v = \Nrm (-)$ (where $v\in \complm $ is the unique element).
    Since we eventually prove Theorem \ref{thm:density} in full strength, we only give a sketch of this reduction argument.
    Take any point $v\in \complm $ and set $X':= \kurve \setminus \br v$ and $\Ded ':= \Gamma (X' ,\mcal O_{\kurve })$.
    We have a canonical injection $i\colon \Ded '\inj \Ded $ and know that
    all but finitely many associate classes (corresponding to a subset of $X'\setminus X $) of prime elements of $\Ded '$ remain prime elements in $\Ded $.
    Since $\Ded '^*=\Fq ^*$, those exceptional prime elements are finite in number, so in particular have density zero in $\mcal P_{\Ded '}$.
    Let $\mcal P_{\Ded '}^-$ be the set of remaining prime elements.
    Now we apply Theorem \ref{thm:density} to $\Ded '$ and $\mcal P_{\Ded '}^-\subset \mcal P_{\Ded '}$
    to find a $k$-dimensional affine subset contained in $\mcal P_{\Ded '}^-$.
    Since the canonical injection $i$ (which is of course $\Fq $-linear) carries $\mcal P_{\Ded '}^-$ into $\mcal P_{\Ded }$,
    we have found a $k$-dimensional affine subset in $\mcal P_{\Ded }$.
    This proves Theorem \ref{thm:density} for $\Ded $ in the special case $A=\mcal P_{\Ded }$.
\end{remark}

\subsection{The choice of a subring}
The constellation theorem \cite{KMMSY} for the ring of integers $\mcal O _K$ of a number field $K$
ensures that the set of prime elements of $\mcal O_K$ contains a $\bbZ $-homothetic copy of any given finite subset $S\subset \mcal O _K$.
In this subsection we choose a subring $\ded \cong \Fq [t] $ of $\Ded $ which plays the role of
$\bbZ $ in $\mcal O_K$.

Recall the definition
$\spDed =\Spec \Ded $ and that $\kurv $ is the complete non-singular curve over $\Fq $
containing $\spDed $ as an open subscheme.
Also $n= \mgn\complm $.

\begin{proposition}\label{prop:choice-of-subring}
    There exists an element $t\in \Ded $ such that the following two conditions are satisfied:
    \begin{enumerate}
        \item\label{item:finite}
        $\Ded $ is finite over $\Fq [t]$, say of rank $\therank $;
        \item\label{item:independent}
        the value $\lnorm t _v $ is independent of $v\in\complm $, say $q^d$.
    \end{enumerate}
    Furthermore we have $\therank = nd $ and
    the canonical linear norms of $\Ded $ and $\ded :=\Fq [t]$ satisfy
    the following compatibility:
    if we write $\lnorm - _{\Ded }$ for the canonical norm of $\Ded $ and $\lnorm -_{\ded }$ for that of $\ded $, then
    we have $\lnorm - _{\ded }^d=\lnorm - _{v}$ as functions on $\ded $ for all $v\in \complm $ and in particular
    $\lnorm - _{\ded }^d = \lnorm - _{\Ded }$.
\end{proposition}
\begin{proof}
%
    Let $d$ be a large enough common multiple of $\deg (v)$ ($v\in\complm $).
    We claim that there exists a function $\phi \in \Ded $ which has a pole at each $v\in\complm $ of order exactly $d/\deg (v)$.
    For this, for each $v$ consider the set $\Gamma (\kurve ,\mcal O_{\kurve }(\frac{d}{\deg v} v ))$ of rational functions on $\kurve $
    whose only possible pole is $v$ with order $\le d/\deg (v)$.
    By the Riemann-Roch formula \cite[Corollary 4 to Theorem 5.4, p.49]{Rosen}, if $d$ is large enough the inclusion
    \begin{equation}
        \Gamma \paren{\kurve ,\mcal O_{\kurve } (\paren{\frac{d}{\deg v}-1} v ) } \quad \subset \quad \Gamma \paren{\kurve ,\mcal O_{\kurve } (\frac{d}{\deg v} v ) }
    \end{equation}
    is a proper one so there is a function $\phi _v$ whose only pole is at $v$ and of order exactly $d/\deg (v)$.
    Choose one such $\phi _v$ for each $v$ with a common $d$.
    Then the function $\phi := \sum _{v\in\complm } \phi _v$ has the claimed property.

    Denote also by $\phi $ the corresponding finite map of curves $\phi\colon \kurve \to \bbP ^1$.
    By the choice of $\phi $ we have an equality of divisors on $\kurve $:
    \begin{equation}\label{eq:pullback-pi}
        \phi ^* (\infty ) = \sum _{v\in\complm } \frac{d}{\deg (v)} v =: D.
    \end{equation}
    One can also see that the degree $\therank $ of the map $\phi $ equals $\deg (D) = nd$.
    Let $t$ be the coordinate of $\bbA ^1 \subset \bbP ^1$.
    We have $\phi\inv (\bbA ^1) = \kurve \smallsetminus |D|=\spDed $
    so $t$ can be seen as an element of $\Ded $ and assertion \eqref{item:finite} holds.

    Next, let $f(t)\in \Fq [t]$ be a polynomial of degree $e$.
    Since the valuation $v_{\infty }\colon \Fq (t)^* \to \bbZ $ at $\infty \in \bbP ^1$ agrees with the degree function when restricted to $\Fq [t]$, we have $\lnorm f _{\ded }:=\lnorm f _{v_\infty } =q^e$.
    By \eqref{eq:pullback-pi} we know $v(f) = - \frac{de}{\deg (v)}$ for each $v\in \complm $. 
    It follows that
    \begin{equation}\label{eq:norm-of-f}
        \lnorm f _v = \paren{ {q^{\deg (v)} } } ^{\frac{de}{\deg (v)}} = q ^{de} = \lnorm f _{\ded }^d
    \end{equation}
    for all $v$.
    Therefore the assertion \eqref{item:independent} and the compatibility assertion holds.
\end{proof}

We fix an $\ded \subset \Ded $ as in Proposition \ref{prop:choice-of-subring} throughout the paper.
To avoid overloaded notation, we will avoid the use of the canonical norm of $\ded $ as much as possible and reserve the symbol $\lnorm -$ for the canonical norm of $\Ded $.
As a consequence we use the following potentially confusing piece of notation:
\begin{equation}\label{eq:confusing}
    \begin{array}{rl}
        \ded _{\le N} &:= \br{ f\in \ded \mid \lnorm f \le N }
        \\
        &=\br{ f\in \ded \mid \lnorm f _{\ded }\le N^{1/d} }
    .\end{array}
\end{equation}
Note that therefore the cardinality $\mgn{\ded _{\le N}}$ is equal, up to a bounded constant, to
$N^{1/d }$.
Despite this potential confusion, this notation is convenient in the bulk of our discussion.


\subsection{Equivalence of linear norms}

Let $\ded \subset \Ded $ be as in Proposition \ref{prop:choice-of-subring}.
We know $\Ded $ is a free $\ded $-module of rank $\therank $.
Let $\alpha _1,\dots ,\alpha _r \in \Ded $ be a basis.
One can consider the {\em max norm} on $\Ded $ with respect to this basis:
\begin{equation}
    \lnorm {\sum _{i=1}^r f_i \alpha _i } _{\bsb \alpha }
    := \max _i \{ \lnorm {f_i} _{\Ded }  \} .
\end{equation}
It is an ultrametric norm on the abelian group $\Ded $. 

Let us recall that two norms $\lnorm - _1$ and $\lnorm - _2$
on an abelian group $G$ are said to be {\em equivalent}
if there are positive real numbers $c,C>0$ such that
the next inequality holds on $G$:
\begin{equation}
    c \lnrom - _2 \quad\le\quad \lnorm - _1 \quad\le\quad C \lnorm - _2 .
\end{equation}
It is easy to see that the equivalence class of the norm $\lnorm - _{\bsb \alpha }$ is independent of the choice of the basis
$\bsb \alpha
= (\alpha _1,\dots ,\alpha _r)$.

\begin{proposition}\label{prop:equivalence-of-norms}
    For any given $\ded $-basis $\bsb \alpha $ of $\Ded $,
    the associated norm $\lnorm - _{\bsb \alpha }$ is equivalent to the canonical norm $\lnorm -  _{\Ded }$.
\end{proposition}
\begin{proof}
    First, if $\lnorm{ f_i}_{\Ded } \le N$ holds for all $i$,
    then by the ultrametricity and submultiplicativity of $\lnorm - _{\Ded }$ we have
    \begin{equation}
        \lnorm {\sum _i f_i \alpha _i } _{\Ded }
        \le N\cdot \paren{\max _i \lnorm {\alpha _i }_\Ded}
    \end{equation}
    so that we have $\lnorm - _{\Ded } \le \lnorm - _{\bsb\alpha }\cdot \paren{\max _i \lnorm {\alpha _i }_\Ded }$.

    The inequality in the other direction is slightly harder.
For a notational reason, let us introduce
the degree function $\deg \colon \Ded \to \bbZ \cup \br{-\infty }$ defined by
\[ \deg (\alpha ) :=  \log _q(\lnorm\alpha _{\Ded }) . \]
For integers $M\ge 0$, denote by $\Ded _{\deg \le M}$ the $\Fq$-vector subspace of elements with degree $\le M$;
of course one has $\Ded _{\deg \le M}=\Ded _{\le (q^M)}$.
Note that by Proposition \ref{prop:choice-of-subring} \eqref{item:independent}, the element $t\in \Ded $ is {\em multiplicative} in the sense that the equality
\begin{equation}
    \lnorm {t\alpha }_{\Ded} = \lnorm t _\Ded \cdot \lnrom {\alpha }_{\Ded },
    \quad \te{i.e., } \deg (t\alpha ) = \deg (t)+\deg (\alpha )
\end{equation}
holds for all $\alpha \in \Ded $ rather than a mere inequality.
(Actually, {\em all} elements of $\ded $ are multiplicative by \eqref{eq:norm-of-f}.)
It follows that the following mutiplication by $t$ map
is injective for all $M\ge 0$, where we recall from Proposition \ref{prop:choice-of-subring} that $d=\deg (t)$:
\begin{equation}\label{eq:mult-by-t}
    \frac{\Ded _{\deg \le M}}{\Ded _{\deg \le M-d}}
    \xrightarrow{\times t}
    \frac{\Ded _{\deg \le M+d}}{\Ded _{\deg \le M}}
.\end{equation}
We claim that it is also surjective for all sufficiently large $M\ge 0$.
There are at least two ways to see this.
One is to use the Riemann-Roch theorem which tells us that both sides
of \eqref{eq:mult-by-t} have the same dimension for $M$ large enough.

The second is more down-to-earth.
Let $M_0:= \max _{1\le i\le \therank } \{ \deg (\alpha _i) \} $ and suppose $M\ge M_0$.
Write an arbitrary element $\alpha \in \frac{\Ded _{\deg \le M+d}}{\Ded _{\deg \le M}}$
in the form
\begin{equation}
    \alpha = \sum _i f_i (t) \alpha _i \mod \Ded _{\deg \le M}, \quad f_i (t)\in \ded .
\end{equation}
Since $\alpha _i \in \Ded _{\deg \le M_0} \subset \Ded _{\deg \le M}$
is zero in the group in question,
we may assume $f_i(t)$ has no constant term.
Then we have a well-defined element $\alpha /t := \sum _i \frac{f_i(t)}t \cdot \alpha _i $ which is in $\Ded _{\deg \le M}$
because $t$ is a multiplicative element.
Then $\alpha $ is the image of $\alpha /t $ under the map \eqref{eq:mult-by-t}.

In any case let $M_0$ be such that \eqref{eq:mult-by-t} is surjective for all $M\ge M_0$.
Now since $\Ded _{\deg \le M_0}$ is a finite set, there trivially exists an $e_0\ge 0$ such that
all $\alpha \in \Ded _{\deg \le M_0}$ can be written in the form
\begin{equation}
    \alpha = \sum _i f_i(t) \alpha _i \quad \te{ with } \deg (f_i(t) ) \le \deg (\alpha )+ e_0 .
\end{equation}
By induction on $\deg (\alpha )$ using the bijection \eqref{eq:mult-by-t}, the same holds for all $\alpha \in \Ded $.
In multiplicative terms, this precisely says there is a positive constant $C=q^{e_0}$ such that
\begin{equation}
    \lnorm \alpha _{\bsb \alpha } \le C \cdot \lnorm \alpha _{\Ded }  \quad\te{ for all }\alpha \in \Ded
.\end{equation}

This complets the proof of Proposition \ref{prop:equivalence-of-norms}.
\end{proof}

For the ``prime elements in an ideal'' case of Theorem \ref{thm:homothetic}, let $\ideala \subset \Ded $ be a non-zero ideal. It is also a rank $\therank $ free $\ded $-module.
We can consider the restriction of the canonical norm $\lnorm - _{\Ded }$ to $\ideala $
and the max norm $\lnorm -_{\boldsymbol{\beta }}$ with respect to an $\ded $-basis $\boldsymbol{\beta }=(\beta _1,\dots ,\beta _r )$ of $\ideala $.
\begin{corollary}\label{cor:equivalence-of-norms-ideal}
    The two linear norms $\lnorm -_{\Ded }$ and $\lnorm -_{\boldsymbol{\beta }}$ on $\ideala $ are equivalent.
\end{corollary}
\begin{proof}
    While the proof of Proposition \ref{prop:equivalence-of-norms} works for this case just as well,
    here we present a proof using the proposition.
    Let $\boldsymbol{\alpha }$ continue to be an $\ded $-basis of $\Ded $.
    By Proposition \ref{prop:equivalence-of-norms}, it suffices to show that the restriction of $\lnorm -_{\boldsymbol{\alpha }}$ to $\ideala $
    and $\lnorm -_{\boldsymbol{\beta }}$ are equivalent.

    Each $\beta _i$ can be written (uniquely) as $\beta _i=\sum _{1\le j\le \therank }g_{ij}\alpha _j$.
    Take a positive number $C$ such that $C\ge \lnorm{g_{ij}}_{\Ded }$ for all $i,j$.
    For an element $x= \sum _{i}f_i\beta _i  = \sum _j \paren{\sum _i f_ig_{ij}} \alpha _j $,
    by the ultrametricity of $\lnorm -_{\Ded }$ and the choice of $C$ we have:
    \begin{equation}
        \lnorm x _{\boldsymbol{\alpha }} = \max _j \br{\lnorm{\sum _i f_ig_{ij}}_{\Ded } }
        \le \max _{i,j} \br{ \lnorm{ f_ig_{ij} }_{\Ded} }
        \le C \max _i \br{ \lnorm{f_i}_{\Ded} }
        = C \lnorm x _{\boldsymbol{\beta}}
    .\end{equation}

    Next, by the theory of finitely generated modules over a principal ideal domain (say),
    we know that there is a non-zero element $f(t)\in \ded =\Fq [t]$ such that $f(t)\Ded \subset \ideala $.
    The previous argument applied to the element $f(t)x \in f(t)\Ded $ gives
    $\lnorm {f(t)x }_{\boldsymbol{\beta }}
        \le
        C\lnorm {f(t)x }_{f(t)\boldsymbol{\alpha }}$.
        By the definition of the max norm we can isolate the $f(t)$-factors so that:
    \begin{equation}
        \lnorm x _{\boldsymbol{\beta }}
        \le
        \frac C{\lnorm{f(t)}_{\Ded }}\cdot
        \lnorm {x }_{\boldsymbol{\alpha }}
    .\end{equation}
    This completes the proof.
\end{proof}

\begin{remark}
    Propositoin \ref{prop:equivalence-of-norms} fails if $\ded $ is not chosen as in Proposition \ref{prop:choice-of-subring} even if $\Ded $ is finite over $\ded $.
    For example, set $\ded := \Fq [t] $ and $\Ded := \Fq [s,s\inv ]$ and consider the homomorphism
    \begin{align}
        \ded &\inj \Ded
        \\
        t & \mapsto \frac{(s-1)^3}{s}
    .\end{align}
    If we let $f$ be the induced finite map $f\colon \bbP ^1\to \bbP ^1$, we have $f^* (\br{ \infty }) = 2\br{\infty }+\br{0}$ as divisors.
    The element $t$ is not a multiplicative element for the canonical norm $\lnorm -$ because for example $\lnorm{t/s}=\lnorm{(s-1)^3/s^2} =q^{2}$ which is not equal to
    $\lnorm t \cdot \lnorm {1/s} = q^{2}\cdot q^{1}$.
    Instead $t$ is a multiplicative element for the following linear norm $\lnorm - '\colon \Ded \to \bbR _{\ge 0}$:
    \begin{equation}
        \lnorm \alpha ' := \max \br{ q^{-v_\infty (\alpha )}, q^{-2v_0 (\alpha )} } .
    \end{equation}
    The arguments of the proof of Proposition \ref{prop:equivalence-of-norms} show that
    the max norm on $\Ded \cong \ded ^3$ is equivalent to $\lnorm - '$.
    However, one easily sees that $\lnorm - '$ is not equivalent to $\lnorm -$;
    for example, one has $\lnorm{ 1/s^n}=q^{n} $ and $\lnorm {1/s^n} '=q^{2n}$ so their ratio is not bounded.
\end{remark}

Recall from \eqref{eq:confusing} that we endow $\ded $ with the induced norm from $\Ded $
and so we have $N^{1/d}<\mgn{\ded _{\le N}} \le q\cdot N^{1/d}$ where the right hand inequality becomes an equality when $N$ is a power of $q^d$.
Let us note the following simple observation.

\begin{lemma}\label{lem:linear-map-inverse-image}
    Let $t,r\ge 1$ be positive integers and $\phi \colon \ded ^t \to \ded ^r$
    a surjective $\ded $-linear map.
    Then there exists a positive number $U>0$ such that for all $N>0$ and $x\in (\ded _{\le N})^r$,
    the following set
    \begin{equation}
        \phi ^{-1} (x ) \cap (\ded _{\le UN})^t
    \end{equation}
    contains at least $\mgn{\ded _{\le N}} ^{t-r}$ elements.
\end{lemma}
\begin{proof}
    Let us denote by $\lnorm - _{\ded ^t}$ and $\lnorm -_{\ded ^r}$ the max norms on $\ded ^t$ and $\ded ^r$ with respect to the standard bases.
    Choose an $\ded$-linear section $\sigma \colon \ded ^r \to \ded ^t$ to $\phi $.
    Let $\bsb e _i\in \ded ^r$ be the standard basis ($1\le i\le r$)
    and set $\bsb f_i := \sigma (\bsb e_i)$.
    We know $\ker (\phi )$ is a free $\ded $-module of rank $t-r$;
    choose a basis $\bsb f_{r+1},\dots ,\bsb f_{t}$ of $\ker (\phi )$.
    Let $U>0$ be larger than $\lnorm{ \bsb f_i}_{\ded ^t} $ for all $1\le i\le t$.
    We claim this $U$ works.

    Suppose we are given $N>0$ and $x\in (\ded _{\le N} )^\therank$.
    Writing $x$ in the form $x=\sum _{i=1}^r a_i\bsb e_i $, we see
    \begin{equation}
        \lnorm {\sigma (x)}_{\ded ^t } =\lnorm{\sum _{i=1}^r a_i\bsb f_i }_{\ded ^t} \le (\max _{i} \lnorm{ a_i})\cdot U \le NU .
    \end{equation}
    Also, for each choice of $a_{r+1},\dots ,a_t \in \ded _{\le N} $
    we have by the same reasoning:
    $    \lnorm{ \sum _{i=r+1}^t a_i \bsb f_i }_{\ded ^t} \le UN.$
    Therefore we get $\lnorm{\sigma (x) + \sum _{i=r+1}^t a_i \bsb f_i }_{\ded ^t}\le UN$
    for all choices of $a_i$'s as above.
    Since this last element is in $\phi ^{-1}(x)$, 
    our claim follows.
\end{proof}

\begin{corollary}\label{cor:equivalence-norms-inverse-image}
    Let $\phi \colon \ded ^t \to \ideala $
    be a surjective $\ded $-linear map.
    Then there exists a positive number $U>0$ such that for all $N>0$ and $x\in \ideala _{\le N}$,
    the following set
    \begin{equation}
        \phi ^{-1} (x ) \cap (\ded _{\le UN})^t
    \end{equation}
    contains at leat 
    $\mgn{\ded _{\le N}}^{t-\therank }$
    elements.
\end{corollary}
\begin{proof}
    This immediately follows from Proposition \ref{prop:equivalence-of-norms}, Corollary \ref{cor:equivalence-of-norms-ideal} and Lemma \ref{lem:linear-map-inverse-image}.
\end{proof}

\subsection{``Geometry of numbers''}
Consider the group homomorphism $\phi \colon \Ded ^* \to \bigoplus _{v\in\complm} \bbZ $ defined by
$\alpha \mapsto ( v(\alpha ) )_v$.
It is known that its image has rank $n -1$:
\begin{equation}\label{eq:rank-of-image}
    \rank \bigpa{\image (\phi ) }=n-1 .
\end{equation}
For a proof, see \cite[Proposition 14.2, p.243]{Rosen}.
For the convenience of the reader
let us recall a proof in the language of sheaves.
Let $j\colon \spDed =\Spec \Ded \inj \kurve $ be the open immersion.
By the short exact sequence of sheaves on $\kurv $
(where $i_v\colon \br{v}\inj \kurve $ is the closed immersion):
\begin{equation}
    1\to \mcal O_{\kurv }^* \to j_*\mcal O^*_{\spDed } \xrightarrow \phi \bigoplus _{v\in\complm} i_{v*} \bbZ \to 0,
\end{equation}
we get an exact sequence $\Fq ^* \inj \Ded ^* \xrightarrow \phi  \bigoplus _v \bbZ \xrightarrow \delta \Pic (\kurv )$.
Since $\Pic (\kurv )= \bbZ \oplus (\text{finite})$
and the connecting map $\delta $ has nontrivial image into the $\bbZ $-part,
the claim \eqref{eq:rank-of-image} follows.

Now consider the map $\mcal L \colon \Ded\nonzero \to \bigoplus _{v\in \complm }\bbR \cong \bbR ^n$ defined by
\begin{equation}
    \mcal L (\alpha ) :=  (\log \lnrom \alpha _v ) _v
.\end{equation}
It is an obvious analog of the multiplicative Minkowski map $\mcal L$ that was also used in \cite[\S 4]{KMMSY}.
By \eqref{eq:product-formula} we know that $\mcal L$ maps the subset $\Ded ^*$ into the hyperplane $H$ of $\bbR ^n$
defined by:
\begin{equation}\label{eq:def-of-H}
    H=\br{ (x_1,\dots ,x_n)\mid x_1+\dots +x_n=0} \subset \bbR ^n
\end{equation}
and by \eqref{eq:rank-of-image} that $\mcal L (\Ded ^* )\subset H$ is a full-rank lattice.

Let us say a subset $\domain \subset \Ded \nonzero $ is {\em norm-length compatible}
if the set
\begin{equation}
    \left\{  \frac{\lnorm \alpha ^n}{\Nrm (\alpha )} \ \middle|\  \alpha \in \domain   \right\} \subset \bbR
\end{equation}
is bounded from above.
Note from \eqref{eq:def-of-canonical-norm} and \eqref{eq:product-formula} that
this set is always bounded from below by $1$.
As usual, a subset $\domain \subset \Ded \nonzero $ is called an {\em $\Ded ^*$-fundamental domain} of $\Ded \nonzero $
(or of $\Ded $ by slight abuse of terminology)
if the composite map $\domain \inj \Ded\nonzero \surj (\Ded\nonzero )/ \Ded ^*$ is a bijection.
As in \cite[\S 4.3]{KMMSY}, the following statement holds.

\begin{proposition}\label{prop:norm-length-compatible-domain}
    There exist norm-length compatible $\Ded ^*$-fundamental domains
    of $\Ded  $.
\end{proposition}
\begin{proof}
    Consider the function $f\colon \bbR ^n \to \bbR $; $(x_1,\dots ,x_n)\mapsto n\max _{i} \{ x_i \}  - \sum _i x_i $.
    It fits into the following commutative diagram:
    \begin{equation}
        \xymatrix{
            \Ded\nonzero \ar[r]^{\mcal L}\ar[d]_{\frac{\lnrom -^n}{\Nrm (-)}} & \bbR ^n\ar[d]_f \\
            \bbR \ar[r]_{\log }& \bbR .
        }
    \end{equation}
    Let $H \subset \bbR ^n$ be as in \eqref{eq:def-of-H} 
    and
    $\pi \colon \bbR ^n \to H$ the projection along the vector $(1,\dots ,1)$.
    Since the value of $f$ is unchanged by the translation along the vector $(1,\dots ,1)\in \bbR ^n$,
    it follows that the norm-length compatibility of a subset $\domain \subset \Ded \nonzero $ is equivalent to the boundedness
    of the non-negatively valued function $(x_1,\dots ,x_n)\mapsto \max _i x_i$ on the set $\pi (\mcal L (\domain )) \subset H$.
    This is equivalent to the boundedness of the set $\pi (\mcal L (\domain ))$ itself;
    recall that one can define the notion of {\em boundedness} of a subset of $\bbR ^n$
    using any choice of a linear norm and the resulting notion is independent of the choice.
    It follows that for any given bounded subset $\parallelo \subset H$, the set $\mcal L \inv (\pi \inv (\parallelo ))$ is a norm-length compatible subset of $\Ded \nonzero $.

    Recall by \eqref{eq:product-formula} and \eqref{eq:rank-of-image} that $\mcal L (\Ded ^*)$ is a full-rank lattice of $ H$.
    Let $\parallelo \subset H$ be any bounded complete set of representatives for the quotient $H/\mcal L (\Ded ^*)$ (say a fundamental parallelogram).
    The inverse image $\pi\inv (\parallelo ) \subset \bbR ^n$ is a complete set of representatives for the quotient $\bbR ^n/\mcal L(\Ded ^*)$.
    By the previous paragraph, the set $\mcal L^{-1}(\pi ^{-1}(\parallelo )) \subset \Ded \nonzero $ is norm-length compatible.
    It is acted on by the group $\Fq ^*=\ker (\mcal L\colon \Ded ^*\to \bbR ^n)$ and the natural map of quotient sets
    $\mcal L^{-1}(\pi ^{-1}(\parallelo )) / \Fq ^* \to (\Ded \nonzero ) / \Ded ^*$
    is a bijection.
    Therefore, any choice of an $\Fq ^*$-fundamental domain of $\mcal L^{-1}(\pi ^{-1}(\parallelo ))$ gives a norm-length compatible
    $\Ded ^*$-fundamental domain of $\Ded $.
\end{proof}

For $\alpha \in \Ded $, one can consider its $\Ded ^*$-orbit $\alpha \Ded ^*$.
We will need the following bound.
For the big-$O$ notation, see Notation at the end of Introduction.

\begin{proposition}\label{prop:number-of-associates}
    For all $N\ge q$ and $\alpha \in \Ded _{\le N}$, we have
    \begin{equation}
        \mgn{(\alpha \Ded ^*)\cap \Ded _{\le N}}
        = O_{\Ded } \Bigpa{(\log N - \frac 1n\log \Nrm (\alpha ) )^{n-1} +1 }.
    \end{equation}
\end{proposition}
    The term ``$+1$'' is there only to cover the rare case $\Nrm (\alpha )=N^n$ (the largest possible). For the application in this paper, its corollary $\mgn{(\alpha \Ded ^*)\cap \Ded _{\le N}}=O_{\Ded }((\log N)^{n-1})$ suffices and saves space.
\begin{proof}
    We may assume $\alpha \neq 0$.
    For a real number $C$, let $\bbR ^n_{\le C}$ be the set of points $(x_1,\dots ,x_n)$ with $x_i\le C$ for all $i$.
    Let $H\subset \bbR ^n$ as in \eqref{eq:def-of-H} and define $\Delta _{\le C}:= \bbR ^n_{\le C}\cap H$, which is an $(n-1)$-dimensional simplex.
    Let us use the symbol $\Delta '_{\le C}$ to denote translations of $\Delta _{\le C}$ inside $H$.
    Write $\Gamma := \mcal L (\Ded ^*)$.
    Sending the set in question by $\mcal L$, 
    we see:
    \begin{equation}
            \mgn{(\alpha \Ded ^*)\cap \Ded _{\le N}}
            = \mgn{\Fq ^* }\cdot \mgn{ (\mcal L(\alpha )+\Gamma )\cap \bbR ^n_{\le \log N} }
        .
    \end{equation}
    So it suffices to bound the size of the set on the right hand side.
By the translation ``$-\mcal L(\alpha )$'' we have a bijection
\begin{equation}
    (\mcal L(\alpha )+\Gamma )\cap \bbR ^n_{\le \log N}
\cong \Gamma \cap (\bbR ^n_{\le \log N} - \mcal L(\alpha ))
.\end{equation}
Next $\bbR ^n_{\le \log N} - \mcal L(\alpha )$ is the translation of $\bbR ^n_{\le \log N-\frac 1n \log \Nrm (\alpha )} $ by the vector
    $-\mcal L(\alpha )+\frac 1n \log \Nrm (\alpha )\cdot (1,\dots ,1) \in H$.
Therefore the intersection $H\cap \bigpa{\bbR ^n_{\le \log N} - \mcal L(\alpha )}$ is a $\Delta '_{\le \log N-\frac 1n \log \Nrm (\alpha )}$.
Since $\Gamma $ is contained in $H$ anyway, we can write
\begin{equation}
    \Gamma \cap (\bbR ^n_{\le \log N} - \mcal L(\alpha ))
    = \Gamma \cap \Delta '_{\le \log N-\frac 1n \log \Nrm (\alpha )}
.\end{equation}
Now since $\Gamma $ is a lattice in $H$, the cardinality of a set of the form $\Gamma \cap \Delta '_{\le C}$
is asymptotically proportional to $C^{n-1}$ as $C\to +\infty $ with error bounded independently of the specific translation $\Delta _{\le C}\leadsto \Delta '_{\le C}$;
see for example \cite[Appendix A]{GT2}.
This completes the proof.
\end{proof}

\section{\Sz 's theorem}\label{sec:Sz}
Let us keep the notation $\Ded $ and $\ded $ from the previous section.
Namely $\Ded $ is a Dedekind domain finitely generated over $\Fq $
in which $\Fq $ is integrally closed
and $\ded =\Fq [t]$ is a subring of $\Ded $ as in Proposition \ref{prop:choice-of-subring}.
In case the reader is interested in the ``prime elements of ideals'' case of Theorem \ref{thm:homothetic}, let $\ideala \subset \Ded $ be a non-zero ideal.
(If not, they can always assume $\ideala =\Ded $.)
Recall that an {\em $\ded $-homothetic copy} of a (finite) subset $S\subset \ideala $
is by definition
a subset of $\ideala $ of the form
 \[ a \cdot S + \beta  = \br{ a\alpha + \beta \mid \alpha \in S  }\]
 with $a \in \ded $ and $\beta \in \ideala $.
 It it said to be {\em non-trivial} if we can take $a\neq 0$.

The next definition is a rather straightforward translation of \cite[Definition 5.3]{KMMSY}.

\begin{definition}\label{def:S-N-rho-o-pseudorandom}
    \begin{enumerate}
        \item\label{item:standard-shape}
        A finite subset $S\subset \ideala $ is said to be a {\em standard shape}
        if it contains $0$ and generates $\ideala $ as an $\ded $-module.
        (The word ``shape'' is used because $\bbZ $-homothetic copies of $S$ in an abelian group are also called {\em constellations with shape $S$} especially in the context of torsion-free abelian groups.)
        In the rest of this Definition, assume $S$ is a standard shape.

        \item\label{item:def-of-linear-maps}
        Write $k=\mgn S$ and give $S$ a numbering $S=\br{ s_1,\dots ,s_{k-1}, s_k=0 }$ for the ease of notation. 
        Let
        \begin{equation}
            \phi _S \colon \ded ^{k-1} \surj \ideala
        \end{equation}
        be the surjective $\ded $-linear map which sends the standard vector $\bsb e_i$ ($1\le i\le k-1$) to $s_i$.

        Let us denote elements of the sum $\ded ^{k}\oplus \ded ^{k}$ by symbols like $(x_i^{\pm } )_{1\le i\le k}$.
        Given an index $1\le j\le k$ and a function $\omega \colon \br{1,\dots ,\hat j,\dots ,k} \to \{ \pm \} $, we obtain an element $(x_i^{\omega (i)})_{i\in \br{1,\dots ,k}\setminus \br j} $ of $\bigoplus\limits _{i\in \br{1,\dots ,k}\setminus \br j}\ded \cong \ded ^{k-1}$.
        Let us call this map the {\it restriction} along $\omega $ and denote it by $\res _\omega \colon \ded ^{k}\oplus \ded ^k \to \ded ^{k-1}$.

        For $1\le j\le k$, we define a surjective $\ded $-linear map
        \begin{equation}
         \psi _{S,j}\colon \bigoplus _{i\in \br{1,\dots ,k}\setminus \br j}\ded  \surj \ideala
        \end{equation}
        as follows:
        we define $\psi _{S,k}:= \phi _S$.
        For $1\le j\le k-1$ we define $\psi _{S,j}(x_1,\dots \widehat{x_j},\dots ,x_k):= x_k s_j+\sum\limits _{i\in \br{1,\dots ,k-1}\setminus \br j } x_i (s_i-s_j)$.

        \item\label{item:def-of-pseudorandom}
        Let $\rho >0$ be a positive real number and
        $N_0>0$ a positive integer.
        A non-negatively valued function $\lambda \colon \Ded \to \bbR _{\ge 0}$ is said to be
        an {\em $(S,N_0,\rho ,\ded )$-pseudorandom measure} if for every choice of the data below:
        \begin{itemize}
            \item a subset $\mathcal B\subset \ded ^{k } $ which is the product of translates of $\ded _{\le N}$ with $N\ge N_0$,
            \item a subset $\Omega \subset \bigsqcup _{1\le j\le k} \br\pm ^{ \br{1,\dots ,\hat j ,\dots ,k} }$
            (namely a set of pairs $(j,\omega )$ of an index $1\le j\le k$ and a function $\omega \colon \br{1,\dots ,\hat j,\dots ,k} \to \br\pm $),
        \end{itemize}
        we have the inequality:
        \begin{equation}\label{eq:S-N-rho-ded-pseudorandom}
            \left|
            \bbE \paren{
                \prod _{(j,\omega )\in \Omega } (\lambda \circ \psi _{S,j}\circ \res _\omega )
                \emidd \mcal B\times \mcal B
                }
                -1 \right|
                \quad < \quad \rho .
            \end{equation}
        \end{enumerate}
\end{definition}

We will need the following form of relative \Sz\ theorem.

\begin{theorem}\label{thm:Sz}
    Let $S\subset \ideala $ be a standard shape and $\delta >0$ a positive real number.
    Then there exist positive real numbers
    $\rho =\rho (\ded ,\ideala ,S, \delta )$ and $\gamma =\gamma (\ded ,\ideala ,S, \delta )>0$
    such that the following holds.
    Let $\lambda \colon \ideala \to \bbR _{\ge 0}$ satisfy the
    $(S,N_0,\rho ,\ded )$-psedorandomness condition for some $N_0\ge \Nrm(\ideala ) q^{2g-1}$
    (see Definition \ref{def:S-N-rho-o-pseudorandom} for this condition).
    Let $N\ge N_0$ and $A\subset \ideala _{ \le N} $.
    Suppose the following two inequalities are satisfied:
    \begin{align}
        \bbE (\lambda \cdot \ichi _A \emid \ideala _{ \le N} ) &\ge \delta ,
        \label{eq:Sz-positive-density}
        \\[3mm]
        \bbE (\lambda ^{k } \cdot \ichi _A \emid \ideala _{ \le N} ) &\le \gamma N
        , \quad\te{ where }k:=\mgn S.
        \label{eq:Sz-smallness}
    \end{align}
    Then the following inequality holds:
    \begin{equation}\label{eq:Sz-conclusion}
        \bbE \paren{ \prod _{s\in S}(\lambda\cdot \ichi _A) (as +\beta )  \emidd (a,\beta )\in \ded _{\le N} \times \ideala _{ \le N}
         } 
         > \gamma
    .\end{equation}
    In particular there exist non-trivial $\ded $-homothetic copies of $S$ contained in $A$.
\end{theorem}
This statement and its proof are an immediate analog of
\cite[Theorem 5.4]{KMMSY}.
Nonetheless we write down the proof for the convenience
of the reader.
For this we have to recall the notion of {\em weighted hypergraphs} and borrow results on them from combinatorics.
We will content ourselves with the following narrower definition than usual.
\begin{definition}
    An {\em $r$-uniform weighted hypergraph} consists of the following data:
    \begin{itemize}
        \item a finite set $J$;
        \item a finite set $V_i$ of vertices given for each $i\in J$;
        \item for each subset $e\subset J$ with cardinality $r$, a weight function $\nu _e \colon \prod _{i\in e}V_i \to \bbR _{\ge 0}$.
    \end{itemize}
\end{definition}
The case where $\nu _e$ have values in $\br{0,1 }$ recovers the notion of an {\em $r$-uniform hypergraph}
by interpreting the value $0$ as ``no $r$-edge'' and $1$ as ``an $r$-edge.''
The case $r=2$ corresponds to classical (weighted, $\mgn J$-partite) graphs.

    Consider the product $\prod _{i\in J} V_i \times \prod _{i\in J} V_i$ and denote its elements by symbols
    like $(x_i^\pm )_i $.
    Paralelly to the above, for a subset $e\subset J$ and a function $\omega \colon e\to \br\pm $
    we get an element $(x_i^{\omega (i)})_{i\in e}\in \prod _{i\in e}V_i$.
    This defines the {\it restriction} along $\omega $, denoted by
    $\res _\omega \colon \prod _{i\in J} V_i \times \prod _{i\in J} V_i\to \prod _{i\in e} V_i$.

    \begin{definition}\label{def:rho-pseudorandom}
            For a positive real number $\rho >0$, an $r$-uniform weighted hypergraph as above is said to be {\em $\rho $-pseudorandom}
            if for all choices of
            a subset $\Omega \subset \bigsqcup\limits _{\begin{subarray}{c}e\subset J\\ \te{with }\mgn e=r \end{subarray}} \br\pm ^e $
            (namely a set of pairs $(e,\omega )$ of a subset $e\subset J$ with $\mgn e=r$ and a function $\omega \colon e\to \br\pm $),
            the following estimate holds:
    \begin{equation}\label{eq:rho-pseudorandom}
        \left|
        \bbE \paren{
            \prod _{(e,\omega )\in \Omega }
                (\nu _e \circ \res _\omega )
            \emidd  
            \prod _{i\in J}V_i \times \prod _{i\in J} V_i
        }
        - 1
        \right|
        <  \rho .
    \end{equation}
\end{definition}

The next theorem is a deep result from combinatorics.

\begin{theorem}[{Relative Hypergraph Removal Lemma \cite{CFZ}}]\label{thm:hypergraph}
    Let $0<r\le k$ be positive integers and $\varepsilon >0$ be a positive real number.
    Then there exist positive real numbers $\gamma =\gamma (r,k,\varepsilon )$ and
    $ \rho =\rho (r,k,\varepsilon )>0$ such that the following holds.

    Let $((V_i)_{i\in J},(\nu _e)_{e\subset J, \mgn e=r})$ be a $\rho $-pseudorandom $r$-uniform weighted hypergraph.
    Suppose given a subset $E_e\subset \prod _{i\in e}V_i$ for each $e\subset J$ with $\mgn e=r$ and
    suppose that the following inequality holds:
    \begin{equation}\label{eq:hypergraph-thm-inequality}
        \bbE \paren{
            \prod _{e\subset J, \mgn e=r} (\nu _e\ichi _{E_e}) ((x_i)_{i\in e})
        \emidd (x_i)_{i\in J} \in \prod _{i\in J} V_i
            }
            \le \gamma .
    \end{equation}
    Then there is a family of subsets $E_e'\subset E_e$ for $e\subset J$ with $\mgn e=r$
    such that:
    \begin{align}
        \label{eq:no-triangle}
        &\bigcap _{e\subset J,\ \mgn e=r} E_e' \times V_{J\setminus e} = \varnothing, \quad\te{ and }
        \\
        \label{eq:few-edges-removed}
        &\bbE (\nu _e \ichi _{E_e\setminus E_e'} \emid V_e) \le \varepsilon \quad\te{ for all }e
    .\end{align}
\end{theorem}
\begin{proof}
    Conlon-Fox-Zhao \cite[Theorem 2.12]{CFZ} state this in a slightly different way but their proof actually shows our current statement.
    See \cite[Theorem 5.10]{KMMSY} for detail.
\end{proof}

Recall that we write $\spDed =\Spec \Ded $ and let $\kurve $ be the complete non-singular curve containing $\spDed $ as an open subscheme.
Also let us recall some integer quantities:
    \newcommand{\integerparameters}{
        n&=\mgn{\complm}, \\
        d_0&=\lcm \br{ \deg (v)\mid v\in\complm  }, \\
        d&=v(t) \quad (\te{independent of }v\in \complm ), \\
        \therank &=nd =\rank _{\ded }(\Ded ), \\
        g&=\te{ the genus of }\kurve
    }
\begin{equation}\label{eq:some-quantities}
    \begin{array}{cl}
        \integerparameters
    .\end{array}
\end{equation}

\begin{proof}[Proof of Theorem \ref{thm:Sz}]
    Let $S\subset \ideala $ and $\delta >0$ be as in the statement.
    Recall $k= \mgn{ S}$.
    Consider the $\ded $-linear map $\phi _S \colon \ded ^{k-1}\surj \ideala $ in Definition \ref{def:S-N-rho-o-pseudorandom}.
    By Corollary \ref{cor:equivalence-norms-inverse-image}, there is a constant $U>0$ such that for every $N\ge 1$
    and $\alpha \in \ideala _{\le N}$, the set $\phi _S\inv (\alpha )\cap \ded _{\le UN}^{k-1}$ contains at least $(N^{1/d })^{k-1-\therank  }$ elements.
    Using Theorem \ref{thm:hypergraph}, we set positive numbers $\varepsilon $, $\gamma $ and $\rho >0$ as:
    \begin{align}\label{eq:def-of-epsilon}
        \varepsilon &:= \delta / (\Nrm (\ideala )k q^{nd_0+g+k-2}U^{(k-1)/d} )  ,
        \\
        \gamma &:= \gamma (k-1,k , \varepsilon ),
        \\
        \rho &:= \rho (k-1,k , \varepsilon )
    ,\end{align}
    whose motivation will only be clear later. 

    Now suppose we are given an $(S,N_0,\rho ,\ded )$-pseudorandom function $\lambda \colon \ideala \to \bbR _{\ge 0}$ with the above $\rho $ and some $N_0\ge 0$.
    Also let $N\ge N_0$ and suppose the subset $A\subset \Ded _{\le N}$ satisfies \eqref{eq:Sz-positive-density} and \eqref{eq:Sz-smallness}.
    Out of these data, we construct a $(k-1)$-uniform weighted hypergraph $( (V_i)_{1\le i\le k},\nu _j\colon \prod _{i\neq j} V_i \to \bbR _{\ge 0} )$
    as follows. The vertex sets are:
    \begin{align}
        V_i &:= \{ \te{hyperplanes }H\subset \ded ^{k-1} \mid
        H \te{ is defined by }x_i = h \te{ with }h \in \ded _{\le UN} \}
         \quad\te{ if }1\le i\le k-1 , \\[2mm]
        V_{k} &:=
        \{ \te{hyperplanes }H\subset \ded ^{k-1} \mid
        H \te{ is defined by } \sum\limits _{i=1}^{k-1} x_i =h
        \te{ with }h\in \ded _{\le UN} \} .
    \end{align}

    To define the weight functions,
    note that for any index $1\le j\le k $ and tuple
    $(H_i)_{i} \in \prod\limits _{i\in \br{1,\dots ,\hat j,\dots ,k} } V_i$,
    the intersection $\bigcap _{i\neq j} H_i$
    consists of exactly one point of $\ded ^{k-1}$.
    Let
    \[ T\colon \bigsqcup _{j=1}^k \paren{\prod _{i\neq j} V_i} \to \ded ^{k-1} \]
    be the map sending a given tuple to this point.
    Its restriction to the $j$th summand shall be denoted by $T_j$ if we need to emphasize the domain of definition.
    By abuse of notation, also denote by $T_j$ the composite map $\prod _{i=1}^k V_i \xrightarrow{\pr _j} \prod _{i\neq j} V_i\xrightarrow{T_j}\ded ^{k-1}$
    where $\pr _j$ is the projection dropping the $j$th entry.
    We define the weight functions for $1\le j\le k $:
    \begin{equation}
        \nu _j \colon
        \prod _{i\neq j} V_i \xrightarrow{T}
        \ded ^{k-1}
        \xrightarrow{ \phi _S}
        \ideala
        \xrightarrow\lambda
        \bbR _{\ge 0}.
    \end{equation}
    We can specify tuples $ (H_i)_{1\le i\le k} $ of hyperplanes by tuples $(h_i)_{1\le i\le k}\in (\ded _{\le UN} )^{k} $ of scalars appearing in their defining equations.
    This gives us the left-hand vertical map in the following commutative diagram,
    where the map $\psi _{S,j}$ was defined in Definition \ref{def:S-N-rho-o-pseudorandom}:
    \begin{equation}
        \xymatrix@C=50pt{
            \prod\limits _{i\in\br{1,\dots ,\hat j, \dots ,k} } V_i \ar[r]^{T}\ar@{^{(}->}[d] & \ded ^{k-1}\ar[d]^{\phi _S} \\
            \bigoplus\limits _{i\in\br{1,\dots ,\hat j, \dots ,k} }\ded \ar[r]_{\psi _{S,j}} & \ideala
        & \te{ for all }j .}
    \end{equation}
    It follows that the estimate \eqref{eq:S-N-rho-ded-pseudorandom} implies the estimate \eqref{eq:rho-pseudorandom} for the weighted hypergraph at hand.
    Thus it is $\rho $-pseudorandom.

    Define a subset $E_j\subset \prod _{i\neq j} V_i$ for each $j$ by:
    \begin{align}
        E_j &:= \{ (H_i)_{i\neq j} \mid  \bigcap _{i\neq j} H_i \subset \phi _S\inv (A) \}
        \\ &= T _j\inv \phi _S\inv (A) ,
    \end{align}
    and set $\wti E _j := E_j \times V_j \subset \prod _{i=1}^{k} V_i $.
    The significance of these sets is as follows:
    given an element $\bm H = (H_i)_i\in \bigcap _{i=1}^{k} \wti E _i $,
    the subset $\{ \phi _S T_i (\bm H ) \} _{i=1}^{k} \subset \ideala $ is an $\ded $-homothetic copy of $S$,
    nontrivial if and only if $\bigcap _{i=1}^{k} H_i =\varnothing $.

    Now, specifying a tuple $(h_i)_{1\le i\le k}\in (\ded _{\le UN})^k$ is equivalent to specifying its first $(k-1)$ entries $\bsb h = (h_i)_{1\le i\le k-1} \in (\ded _{\le UN})^{k-1}$ and a scalar $a=h_{k} - \sum _{i=1}^{k-1} h_i \in \ded _{\le UN}$.
    For $1\le i\le k-1$, let
    \begin{equation}
        \kihon _i = (0,\dots ,0,1,0,\dots ,0) \in \ded ^{k-1}
    \end{equation}
    be the $i$-th standard vector and
    let $\kihon _{k} =0 $ be the zero vector.
    We claim the following inequality:
    \begin{equation}\label{eq:we-claim-the-following-inequality}
        \bbE \paren{\prod _{i=1}^{k} (\lambda\ichi _A )(\phi _S ( \bsb h +  a \bsb e_i   )) \emidd (\bsb h,a )\in (\ded _{\le UN})^{k-1}\times (\ded _{\le UN}\nonzero )  }
        > \gamma .
    \end{equation}
    Toward contradiction, suppose otherwise.
    The assumption \eqref{eq:Sz-smallness} is equivalently formulated as
    \begin{equation}
        \bbE \paren{\prod _{i=1}^{k} (\lambda\ichi _A )(\phi _S ( \bsb h +  a \bsb e_i   )) \emidd (\bsb h,a )\in (\ded _{\le UN})^{k-1}\times \{ 0 \}  }
        \le \gamma
    .\end{equation}
It follows that the expectation computed on $(\ded _{\le UN})^{k-1}\times \ded _{\le UN}$ is also $\le \gamma $.
By the definitions of $\nu _j$ and $E_j$ as pullbacks, we know that the following commutes:
\begin{equation}
    \xymatrix@C=50pt{
    \prod\limits _{i\neq j}V_i \ar[r]^{\phi _ST}\ar[dr]_{\nu _j\ichi _{E_j}} &\ideala \ar[d]^{\lambda \cdot \ichi _A}
    \\
    &\bbR _{\ge 0}.
    }
\end{equation}
Also recalling the definition of $T$ we find that this last inequality precisely says
that the hypothesis \eqref{eq:hypergraph-thm-inequality} of Theorem \ref{thm:hypergraph} is satisfied for our situation.
Therefore there is a family of subsets $E_i'\subset E_i $ as in Theorem \ref{thm:hypergraph}.

We claim that the existence of such $E_i'$ leads to the negation of \eqref{eq:Sz-positive-density}.
Define a map
\begin{align}
    \iota _0 \colon &\phi _S\inv (A) \cap \ded _{\le UN}^{k-1}\to \prod _{i=1}^k V_i \quad\te{ by }
    \\
    &\bsb a \mapsto (\te{the hyperplane }\in V_i\te{ passing through }\bsb a )_i.
\end{align}
We have an equality of maps
$T_j \circ \iota _0 =\id $
from $\phi _S\inv (A)\cap \ded _{\le UN}^{k-1}$ to itself for all $j$.
It follows $\iota _0 $ maps into $\bigcap _{j=1}^k \wti{ E}_j$ (recall $\wti E_j:=E_j\times V_j \subset \prod _{i=1}^k V_i $).
Endow this set with the following filtraion;
for the sake of space, we write
also $\wti E_i':= E_i'\times V_i$:
\begin{equation}
    \bigcap _{i=1}^k \wti E_i
    \supset
    \dots
    \supset
    \quad
    \paren{\bigcap _{i=1}^l \wti E_i} \cap \paren{\bigcap _{i=l+1}^k \wti E_i'}
    \quad
    \supset
    \dots
    \supset
    \bigcap _{i=1}^k \wti E_i' =\varnothing .
\end{equation}
Hence $\bigcap _i \wti E_i$ is the disjoint sum of the successive complements so we can define
a map $\pr  \colon \bigcap _{i=1}^k \wti E_i \to \bigsqcup _{l=1}^k (E_l \setminus E_l')$ by the condition:
\begin{align}
    \te{the restriction }\pr \colon
    \paren{\bigcap _{i=1}^{l-1} \wti E_i} \cap \Bigpa{\wti E_l \setminus \wti E_l' }
    \cap\paren{\bigcap _{i=l+1}^k \wti E_i'}
    \to E_l\setminus E_l'
    \te{ is the projection }\pr _l.
\end{align}
Then define a map $\iota \colon \phi _S\inv (A)\cap \ded _{\le UN}^{k-1} \to \bigsqcup _{i=1}^k (E_i\setminus E_i')$ by $\pr \circ \iota _0$.
We see that $T\circ \iota =\id $ and in particular $\iota $ is injective.
So far we have obtained the following commutative diagram:
\begin{equation}
    \xymatrix{
     &\phi_S\inv (A)\cap \ded _{\le UN}^{k-1}\ar[dl]_{\phi _S}\ar@{_{(}->}[d]_{\iota }
    \\
        A \ar[d]_{\lambda}  & \bigsqcup\limits _{i=1}^k  E_i \setminus  E_i' \ar[l]_{\phi_S T}\ar[dl]^{\bigsqcup _i \nu _i}
    \\
    \bbR _{\ge 0}
        }
,\end{equation}
where we know the fibers of the map $\phi _S$ have cardinality $> (N^{1/d })^{k-1-\therank }$ by the choice of $U>0$.
It follows that
\begin{equation}
    \sum _{\begin{subarray}{c}
        \bm H\in \bigsqcup _i \paren{E_i\setminus E_i'}
    \end{subarray}}
    \nu _i (\bm H)
    \ge
    \sum _{\bsb a\in \phi_S\inv (A)\cap \ded_{\le UN}^{k-1}}
    \lambda (\phi_S(\bsb a))
    >
    (N^{1/d })^{k-1-\therank }\cdot \sum _{\alpha \in A }  \lambda (\alpha ) .
\end{equation}
By \eqref{eq:few-edges-removed} we know that the left hand side is bounded by
\begin{equation}
    \le \sum _{i=1}^k \varepsilon \cdot \mgn{ \ded _{\le UN}}^{k-1} \le k\varepsilon \cdot  (q(UN)^{1/d })^{k-1}.
\end{equation}
By these inequalities and the fact \eqref{eq:size-of-ideala-N} (or \eqref{eq:size-of-Ded-N}) that $\mgn{\ideala _{\le N}} \ge N^{n}/(\Nrm (\ideala )q^{nd_0+g-1})$ for $N\ge \Nrm (\ideala ) q^{2g-1}$ we get:
\begin{equation}
    \bbE ( \ichi _A \lambda \emid \ideala _{\le N})
    < \Nrm (\ideala ) k q^{nd_0+g+k-2} (U^{1/d })^{k-1}\cdot \varepsilon \overset{\eqref{eq:def-of-epsilon} }{=} \delta ,
\end{equation}
contradicting \eqref{eq:Sz-positive-density}.
This shows the claimed inequality \eqref{eq:we-claim-the-following-inequality}.

To conclude, let us deduce \eqref{eq:Sz-conclusion} from \eqref{eq:we-claim-the-following-inequality}.
First, if the elements $a\in \ded $ and $\beta \in \ideala _{\le N} $ satisfy $as+\beta \in A \subset \ideala _{\le N}$ for an $s\in S\nonzero $,
it is necessarily true that $as\in \ideala _{\le N}$ because $\ideala _{\le N}$ is a subgroup.
Since $\lnorm s \ge 1$, we necessarily have $\lnorm a \le N$.
It follows that the terms with $\lnorm a > N$ do not contribute to the expectation \eqref{eq:we-claim-the-following-inequality} so that we obtain
\begin{equation}
    \bbE \paren{\prod _{i=1}^k (\lambda \ichi _{A})(\phi _S (\bsb h +a\bsb e_i ))
    \emidd (\bsb h ,a)\in \ded _{\le UN}^{k-1} \times (\ded _{\le N}\nonzero ) }
    >\gamma .
\end{equation}
Note that the fibers of
$\phi _S \colon \ded _{\le UN}^{k-1} \to \ideala $ have cardinality $\le \mgn{\ded _{\le UN}} ^{k-1-\therank }$
and so the same is true for the vertical map in the next commutative diagram:
\begin{equation}
    \xymatrix@C=4cm{
        \ded _{\le UN}^{k-1} \times \ded _{\le N}
        \ar[d]_{\phi_S\times \id  }
        \ar[rd]^{\hspace{2cm}(\bsb h,a)\mapsto \prod _{i=1}^k (\lambda\ichi _A )(\phi_S(a\bsb e_i + \bsb h )) }
        \\
        \ideala \times \ded _{\le N} \ar[r]_{(\beta ,a)\mapsto \prod _{s\in S} (\lambda \ichi _A)(as+\beta )}
        & \bbR _{\ge 0}.
    }
\end{equation}
It follows that
\begin{equation}
    \mgn{\ded _{\le UN}}^{k-1-\therank }
    \sum _{\begin{subarray}{c}
        \beta \in \ideala _{\le N}
        \\
        a\in \ded _{\le N}
    \end{subarray}}
    \prod _{s\in S}
    (\lambda \ichi _A)(\beta +as)
    \ge
    \sum _{\begin{subarray}{c}
        \bsb h \in \ded _{\le UN}^{k-1}
        \\
        a\in \ded _{\le N}
    \end{subarray}}
    \prod _{s\in S}
    (\lambda \ichi _A)(\phi _S(\bsb h) +as)
.\end{equation}
If we divide both sides by $\mgn{\ded _{\le UN}}^{k-1}\cdot \mgn{ \ded _{\le N}}$ this precisely says:
\begin{equation}
    \frac{\mgn{\ideala _{\le N}} }{\mgn{\ded _{\le UN} }^\therank  }\cdot
    \te{L.H.S. of \eqref{eq:Sz-conclusion}}
        \ge
    \te{L.H.S. of \eqref{eq:we-claim-the-following-inequality}}
    > \gamma .
\end{equation}
Note by \eqref{eq:size-of-ideala-N} that the first factor on the left hand side is smaller than $1$ at least if $g\ge 1$ or $U\ge q$.
It follows that the asserted inequality \eqref{eq:Sz-conclusion} holds.
This completes the proof of Theorem \ref{thm:Sz}.
\end{proof}

\section{The von Mangoldt function}\label{sec:GY}

The definitions and results are parallel to those in \cite[\S 6]{KMMSY}.

\begin{definition}\label{def:von-Mang-function}
    Let $\chi \colon \bbR \to [0,1]$ be a $C^\infty $ function with support $\subset [-1,1]$
    and $\chi (0)=1$, which we fix throughout the paper.
    Let $R>0$ be a positive number. In the proof of the main result we will need to take $R$ very large.
    Let $\Ideals $ be the multiplicative monoid of non-zero ideals of $\Ded $.
    We define the function $\Mang \colon \Ideals  \to \bbR $ by
    \begin{equation}
        \Mang (\idealb ) := \log R \sum _{\begin{subarray}{c} \idealc\in\Ideals  \\ \te{with }\idealc |\idealb \end{subarray}}
        \mu (\idealc ) \chi \left( \frac{\log \Nrm (\idealc )}{\log R } \right) .
    \end{equation}
\end{definition}

Given a non-zero ideal $\ideala $ of $\Ded $, we define the function
$\Manga \colon \ideala \to \bbR $ by the composition:
\begin{equation*}
    \Manga  \colon \ideala \xrightarrow{\alpha \mapsto \alpha \ideala ^{-1}} \Ideals
    \xrightarrow{\Mang} \bbR .
\end{equation*}
Note that the membership $\alpha \in \ideala $ implies that $\alpha \ideala ^{-1}$ is a (non-zero)
ideal of $\Ded $ so that the above composition is well defined.

Below we use the notation $\Nrm (-)$ also for non-zero ideals of $\ded $.
This does not cause confusion because a non-zero ideal of $\ded $ is never an ideal of $\Ded $
    and vice versa.
Of course every non-zero ideal of $\ded =\Fq [t] $ is a principal ideal $f\ded $
and if $f$ has degree $e$ as a polynomial then we have $\Nrm (f\ded )=q^e $ ($=\lnorm f _{\ded }$).
For $N>0$ we have $\ded _{\le N^d} = \br{ f\in \ded \mid \Nrm (f\ded ) \le N }$ by Proposition \ref{prop:choice-of-subring}.
    For non-zero ideals $\idealb \subset \Ded $, by the canonical injection $\ded /\ded \cap \idealb \inj \Ded /\idealb $, we know $\Nrm (\ded \cap \idealb )\le \Nrm (\idealb )$.

    For $\idealb \in \Ideals $, define $\totient (\idealb ):= \mgn{ (\Ded /\idealb )^* }$.
    By Chinese Remainder Theorem we know $\totient (\idealb )= \Nrm (\idealb )\prod _{\idealp |\idealb }\paren{ 1- \frac{1}{\Nrm (\idealp )} } $
    where $\idealp $ runs through the prime ideal divisors of $\idealb $.
    For elements $\alpha \in \Ded \nonzero $ let us write $\totient (\alpha ):= \totient (\alpha \Ded )$.

\begin{theorem}[of Goldston-Y{\i}ld{\i}r{\i}m type]\label{thm:Goldston-Yildirim}
    Let $\ideala $ be a non-zero ideal of $\Ded $.
    Let $m,t\ge 0$ be non-negative integers.
    Let
    \[ \phi _1,\dots ,\phi _m \colon \ded ^t \to \ideala \]
    be $\ded $-linear maps
    whose cokernels are finite and such that $\ker (\phi _i)$ does not contain $\ker (\phi _j)$ whenever $i\neq j$.

    Then there are large positive numbers $R_0> 1$, $w_0> 1$ and a small one $0< f_0< 1$ 
    such that
    for every choice of the quantities below:
    \begin{itemize}
        \item real numbers $R\ge R_0$ and $w\ge w_0$ such that $\frac{\log_q w}{\sqrt {\log R}} < f_0$,
        \item an element $W\in \ded $ whose prime factors are exactly
        $\{ \pi \text{ prime element }\mid \Nrm (\pi\ded ) \le w \} $,
        \item a translate $B\subset \ded ^t $ of a product $\ded _{\le N_1}\times \dots \times \ded _{\le N_t}$ with $N_i\ge R^{2m}/q$ for all $1\le i\le t$,
        \item $b_1,\dots ,b_m\in \ideala \setminus \bigcup _{\idealp |W} \idealp \ideala $,
    \end{itemize}
    we have an estimate
    \begin{multline}\label{eq:GY-statement}
        \Ex \left( \prod _{i=1}^m \Manga (W\phi _i (x)+b_i)^2  \emidd
        x\in B
        \right)
        \\
        = (1+O_{\chi ,m,\therank }\paren{\frac{1}{w \log_q w} }+O_{\chi ,m,\therank ,\Ded }\paren{\frac{\log_q w}{\sqrt{\log R}} } ) \paren{\log R\cdot \frac{\Nrm (W\Ded ) }{\totient (W)}\frac{C_{\chi }}{\kappa _\Ded } } ^m
        ,\end{multline}
        where $C_\chi $ and $\kappa _{\Ded }$ are positive constants associated with $\chi $ and $\Ded $ which are defined in \eqref{eq:def-of-C-chi} and \eqref{eq:pole-of-zeta}. 
    \end{theorem}
    That the error terms can be taken entirely independent of the choice of $B$, $b_1,\dots ,b_m$
    is part of the statement of Theorem \ref{thm:Goldston-Yildirim}.
    Also, the ratio $\Nrm (W\Ded )/\totient (W) = \prod _{\idealp | W} \paren{1-\frac 1 {\Nrm (\idealp )}}$ depends only on the real number $w$ and not on the specific $W$.

    The proof of this theorem occupies the rest of this section.

    Toward the proof of Theorem \ref{thm:Goldston-Yildirim}, first we unfold the relevant definitions to get
\begin{multline}\label{eq:unfolded-expression}
    \text{L.H.S. of \eqref{eq:GY-statement}}
    =
    \\
    \sum _{(\idealb _i ,\idealc _i )_i \in \Ideals ^{2m}} (\log R)^{2m}
    \mu (\idealb _i )\mu (\idealc _i ) \chi \paren{\frac{\log \Nrm (\idealb _i )}{\log R}}
    \chi \paren{\frac{\log \Nrm (\idealc _i )}{\log R} }
    \E \paren{\prod _{i=1}^m \ichi _{\ideala \cdot (\idealb _i\cap \idealc _i)}(W\phi _i(x)+b_i ) \, \middle| \, x\in B }
.\end{multline}
Note that only those terms with $\Nrm (\idealb _i )\le \log R$
and $\Nrm (\idealc _i )\le \log R$ for all $i$ contribute to the sum
because $\Supp \chi \subset [-1,1]$.
Define an $\ded $-linear map $\bar\phi $ by
\begin{equation}\label{eq:def-of-bar-phi}
    \bar\phi \colon \ded ^t \to \bigoplus _i \ideala /\ideala \cdot(\idealb _i \cap \idealc _i) ;\quad
    x \mapsto \bigpa{\phi _i (x)\mod \ideala \cdot (\idealb _i\cap \idealc _i) } _i .
\end{equation}
For a given tuple of ideals $\idealtuple $, let $I\subset \ded $ be the ideal $I=\ded \cap \left( \bigcap _i \idealb _i\cap \idealc _i\right) $. The map $\bar\phi $ factors through $(\ded /I)^t$.
Also we write $ b=(b_i)_i \in \bigoplus _{i=1}^m \ideala $ and $\bar b$ for its residue class in $\bigoplus _i \ideala /\ideala \cdot (\idealb _i\cap \idealc _i)$.
Then for $x\in \ded ^t$, the condition that
\begin{equation}
    F(x):=\prod _i \ichi _{\ideala\cdot (\idealb _i\cap \idealc _i)}(W\phi _i (x)+b_i)=1
\end{equation}
is equivalent to the equality $ W\bar\phi (x)+\bar b =0 $,
namely the equality $\bar F(x):=\ichi _{\br 0} (W\bar\phi (x)+\bar b) =1$.
It follows that we have a commutative diagram:
\begin{equation}\label{eq:comm-diagram-F-F_bar}
    \xymatrix{
        \ded ^t \ar[r]\ar@(ur,ul)[rrr]^{F} & (\ded /I)^t \ar[rr]_{\bar F } & & \br{0,1}
    }
.\end{equation}
The next assetion paves the way for the computation of the $\bbE (-)$ term in \eqref{eq:unfolded-expression}.

\begin{lemma}
    Let $\idealtuple \in \Ideals ^{2m}$ be a tuple with $\Nrm (\idealb_i ),\Nrm (\idealc _i) \le \log R$ for all $i$.
    Let $I$, $F$ and $\bar F$ as in \eqref{eq:comm-diagram-F-F_bar}.
    Then we have an equality
    \begin{equation}\label{eq:expectation-ichi}
        \E \paren{ F \emidd B  }
    =
    \E \paren{ \bar F \emid (\ded /I)^t } .
    \end{equation}
\end{lemma}
\begin{proof}
    Note that $\Nrm (I) \le \Nrm (\bigcap _i (\idealb _i \cap \idealc _i ))\le \prod _i \Nrm (\idealb _i)\cdot \Nrm (\idealc _i) \le R^{2m}$.
    As an elementary fact in $\ded =\Fq [t]$, we know that for any $l\ge \Nrm (I)/q$ the composite map $\ded _{\le l^d}\inj \ded \surj \ded / I $ is surjective.
    (In other words, if $f\in \Fq [t]$ is a polynomial of degree $e$, the polynomials of degree $\le e-1$ form a set of representatives for the quotient ring $\Fq [t]/(f)$.)
    It follows that the composite homomorphism $B\inj \ded ^t \surj (\ded /I)^t$ is surjective.
    By the commutative diagram \eqref{eq:comm-diagram-F-F_bar} our claim follows.
\end{proof}
Now let us move on to compute the right hand side of \eqref{eq:expectation-ichi}.

\subsection{The $\pi $-parts}
Let $\pi \in \ded $ be a prime element.
Let us call a non-zero ideal $\idealb $ of $\Ded $ a {\em $\pi $-ideal} if the $\ded $-module $\Ded /\idealb $ is annihilated by a power of $\pi $.
By the prime decomposition of ideals of $\Ded $, every $\idealb \in \Ideals $ is uniquely written as a product
\begin{equation}
    \idealb = \prod _\pi \idealb ^{(\pi )}
\end{equation}
where $\pi $ runs through the associate classes of prime elements of $\ded $ and $\idealb \ppart $ is a $\pi $-ideal.
We call $\idealb\ppart $ the {\em $\pi $-part} of $\idealb $.
The {\em $\pi $-part} of a tuple of ideals $\idealtuple =(\idealb _i ,\idealc _i)_{1\le i\le m}\in \Ideals ^{2m}$
shall mean the tuple of the $\pi $-parts of its entries:
$\idealtuple \ppart := (\idealb _i\ppart ,\idealc _i\ppart )_{1\le i\le m}$.

Let us write
\begin{equation}
    \myep
\end{equation}
for the quantity in \eqref{eq:expectation-ichi}.
It depends also on the data of $\phi _i$ but we do not include it in the notation.

\begin{lemma}\label{lem:E-is-multiplicative}
The quantity $\myep $ decomposes into the product of its
$\pi $-parts; namely,
\begin{equation}
    \myep = \prod _{\pi } \pep ,
\end{equation}
where $\pi $ runs through the associate classes of the prime elements of $\ded $.
\end{lemma}
\begin{proof}
    By Chinese Remainder Theorem, the map $\bar\phi \colon (\ded /I)^t \to \bigoplus _i \ideala /\ideala (\idealcap ) $ decomposes to its $\pi $-parts,
    that is: 
    \begin{equation}
        \bar\phi = \prod _{\pi } \bar\phi\ppart \colon
        \prod_\pi (\ded /I\ppart )^t
        \to
        \prod _\pi \bigoplus _i \ideala /\ideala (\idealb_i\ppart \cap \idealc_i\ppart ) ,
    \end{equation}
    where $\bar\phi \ppart $ is the map defined by \eqref{eq:def-of-bar-phi} with $\idealtuple \ppart $
    in place of $\idealtuple $.
    Then in \eqref{eq:expectation-ichi}, the $\br{0,1}$-valued function $\bar F$ decomposes into the product of
    functions $\bar F\ppart \colon \ded /I\ppart \to \br{0,1}$ which are defined exactly as $\bar F$ with $\idealtuple \ppart $ in place of $\idealtuple $.
    By a Fubini type computation our assertion now follows.
\end{proof}

Our next task is to evaluate $\pep $.
The computation is divided to two cases: when $\lnorm \pi _{\ded }$ is small and when it is large.
Let us use Greek letters $\alpha ,\beta ,\gamma $ to denote ideals when they are assumed to be $\pi $-ideals for a fixed prime element $\pi \in \ded $.

\begin{lemma}\label{lem:pi-local-computation}
    Consider a tuple of ideals $\pidealtuple = (\pidealb _i,\pidealc _i)_{i}\in \Ideals ^{2m}$ and
    suppose $\pidealb _i$ and $\pidealc _i $ are all $\pi $-ideals for a common $\pi $.
    Then, the quantity $\epp $ equals $1$
    if $(\pidealb _i,\pidealc _i)=(\Ded ,\Ded )$ for all $i$.
    Assuming otherwise in the following,
    we have:
    \begin{enumerate}
        \item\label{item:smaller-than-w} If $\Nrmpi  \le w$, then $\epp =0$.
        \item\label{item:larger-than-w} Assume $\Nrmpi >w$ and $w$ is large enough depending on $\br{\phi _i}_i$. Then:
        \begin{enumerate}
            \item\label{item:one-nontriv-index} if there is only one $i$ with $\pidealb _i \cap \pidealc _i\subsetneq \Ded $,
            then $\epp = \frac 1 {\Nrm (\pidealb _i\cap \pidealc _i)}$.
            \item\label{item:multiple-nontriv-indices} if there are at least two $i$'s with $\pidealb _i \cap \pidealc _i\subsetneq \Ded $,
            then $\epp \le \frac{1}{\Nrm (\pi \ded )^2 }$.
        \end{enumerate}
    \end{enumerate}
\end{lemma}
\begin{proof}
    Let us consider the case \eqref{item:smaller-than-w}.
    In this case, we know $W\in \pi\ded $ by our assumption on the prime factors of $W$ in Theorem \ref{thm:Goldston-Yildirim}.
We claim the value $W\bar\phi (x)+\bar b $ is never $0$ in $\bigoplus _i \ideala / \ideala (\pidealcap )$.
Indeed, choose any $i$ with $\pidealcap \subsetneq \Ded $ (which we are assuming to exist) and any of its prime factors $\idealp $.
It is a prime ideal over $\pi\ded $, and hence $W\in \idealp $.
It follows that $W\phi _i (x)\in \idealp \ideala $ for all $x\in \ded ^t$.
Meanwhile $b_i\notin \idealp\ideala $ by assumption. It follows $W\phi _i (x)+b_i \notin \idealp\ideala $ and in particular $\notin (\pidealcap )\ideala  $.
Our claim follows.

    Next we consider the case \eqref{item:larger-than-w}. In this case the ideals $\pidealb _i,\pidealc _i$ are all coprime to $W $ in $\Ded $.
For the case \eqref{item:one-nontriv-index}, it suffices to show that the map $W\phi _i (-)+b_i \colon (\ded /I )^t \to \ideala /\ideala (\pidealcap )$ is surjective.
Since the translation $+b_i$ and the multiplication by $W$ map on $\ideala / \ideala (\pidealcap )$ are both bijective,
it suffices to show that the map
\begin{equation}\label{eq:linear-map-to-be-shown-surjective}
     \ded ^t \xrightarrow{\phi _i} \ideala \surj \ideala / \ideala (\pidealcap )
\end{equation}
is surjective when $w$ is large enough.
By assumption $\coker (\phi _i) $ is an $\ded $-module which is a finite abelian group.
Hence there are only finitely many prime ideals $(\pi )$ of $\ded $ satisfying $\pi \cdot\coker (\phi _i ) \subsetneq \coker (\phi _i)$.
Now assume $w$ exceeds the norms of those $\pi $'s.
Then as long as $\pidealb _i$ and $\pidealc _i$ are $\pi $-ideals and $\Nrmpi >w$, we have $(\pidealcap )\cdot\coker (\phi _i)=\coker (\phi _i)$,
i.e., the map \eqref{eq:linear-map-to-be-shown-surjective} is surjective.

Let us consider the case \eqref{item:multiple-nontriv-indices}.
First we specify how large $w$ should be.
We are assuming that $\ker (\phi _i)$'s do not contain each other.
For each pair of distinct indices $i,j$, choose an element
$x_{ij}\in \ker (\phi _i)\setminus \ker (\phi _j)$.
Since $\phi _j(x_{ij})\in \ideala $ is non-zero, there are only finitely many prime ideals $\idealp \subset \Ded $ with $\phi _j(x_{ij})\in \idealp\ideala $.
Let $w$ exceed the norms of all the $\idealp$'s appearing this way.

To show \eqref{item:multiple-nontriv-indices} it suffices to verify that the image of the map
$\bar\phi $ 
has cardinality $\ge \Nrm (\pi \ded )^2$.
Suppose $i\neq j$ are among the indices with $\pidealcap \subsetneq \Ded $
and let $\idealp _i ,\idealp _j$ be prime ideals containing them.
We show that the image of the next further composition
\begin{equation}\label{eq:linear-map-with-large-image}
    \ded ^t \xrightarrow{\bar\phi } \bigoplus _{i=1}^m \ideala /\ideala (\pidealcap )
    \surj  (\ideala /\ideala \idealp _i) \oplus (\ideala /\ideala \idealp _j)
\end{equation}
has cadinality $\ge \Nrm (\pi \ded )^2$.
The images of the two elements $x_{ij},x_{ji}$ are respectively the residue classes of $(0,\phi _j (x_{ij}))$ and $(\phi _i(x_{ji}),0)$,
and both are non-zero by the very choice of $w$.
It follows that their $\ded /(\pi )$-linear combinations are all distinct (note that the target is an $\ded / (\pi )$-vector space).
Therefore the image of the map \eqref{eq:linear-map-with-large-image} contains at least $ \mgn{\ded /(\pi )}^2$ distinct elements.
\end{proof}

Now we want to plug our results here into \eqref{eq:unfolded-expression},
but to proceed further, we need the help of Fourier analysis.

\subsection{Fourier transform}

Let $\chihat $ be the Fourier transform of the function $x\mapsto e^x \chi (x)$ so that by inverse Fourier transform:
\begin{equation}
    e^x \chi (x) = \int _{\bbR } \chihat (\xi ) e^{\xi x \kyo } d\xi .
\end{equation}
It follows that $\chi (\frac{\log \Nrm (\idealb )}{\log R}) = \int _{\bbR } \Nrm (\idealb )^{\frac{1}{\log R} (-1+\xi \kyo ) } \chihat (\xi )d\xi $ for ideals $\idealb $.
By the theory of Fourier analysis, we know that $\chihat $ decays rapidly:
\begin{lemma}\label{lem:Fourier-textbook}
    For any given positive numbers $A$ and $b\ge 1$, we have
    $\chihat (\xi ) = O_{A,\chi } (1+|\xi |^{-A})$ and hence
    \begin{align}
        \int _{\bbR } |\chihat (\xi )| d\xi < \infty
        \quad\te{ and }\quad
        \paren{\int _{-\infty }^{-b} + \int _{b}^{+\infty } }|\chihat (\xi ) |d\xi = O_{A,\chi }(b^{-A}) .
    \end{align}
\end{lemma}
\begin{proof}
    See any textbook on Fourier analysis or \cite[Lemma 6.15 and its corollary]{KMMSY}.
\end{proof}

The right hand side of \eqref{eq:unfolded-expression} is written as:
\begin{multline}\label{eq:expression-after-Fourier}
    \sum _{(\idealb _i ,\idealc _i )_i \in \Ideals ^{2m}} (\log R)^{2m}
    \prod _{i=1}^m \mu (\idealb _i)\mu (\idealc _i)
    \chihatint{\idealb _i}{\xi _i} \chihatint{\idealc _i}{\eta _i}
    \\
    \cdot \myep .
\end{multline}

Let $I$ be the interval $I:=[-\sqrt{\log R},\sqrt{\log R}]$.
For tuples $(\undl\xi ,\undl\eta )\in I ^{2m}$, 
consider the infinite sum
\begin{multline}\label{eq:consider-the-infinite-sum}
    \sumXiEta :=
    \\
    \sum _{(\idealb_i,\idealc _i)_i \in (\Ideals )^{2m}}
    \myep \cdot
    \prod _{i=1}^m \mu (\idealb _i )\mu (\idealc _i)
    \Nrm (\idealb _i)^{\frac{1}{\log R}(-1+\xi_i\kyo ) }
    \Nrm (\idealc _i)^{\frac{1}{\log R}(-1+\eta_i\kyo ) }
.\end{multline}

This subsection is devoted to the proof of:
\begin{proposition}\label{prop:convergence-and-estimate}
    The sum $E((\undl\xi ,\undl\eta ),R,w,b)$ converges absolutely
    and uniformly in $(\undl\xi ,\undl\eta )\in \bbR ^{2m}$.
    For any given $A>0$, the quantity \eqref{eq:expression-after-Fourier} is equal,
    up to an error $\pm O_{A,\chi ,m,\therank } ((\log R)^{-A})$, to:
    \begin{equation}\label{eq:convergence-and-estimate}
        (\log R)^{2m}\int _{I^{2m}} \sumXiEta \paren{\prod _{i=1}^m \chihat (\xi _i )\chihat (\eta _i) }d\undl\xi d\undl\eta .
    \end{equation}
\end{proposition}

\subsubsection{Convergence}
Note that by the presence of the \Mobius\ function,
only those terms where all $\idealb_i$ and $\idealc _i $ are square-free contribute to the sum
\eqref{eq:consider-the-infinite-sum} and
that the sum decomposes into the product of its $\pi $-parts by the multiplicativity of the functions involved (see Lemma \ref{lem:E-is-multiplicative} for the multiplicativity of $\myep $).
Namely:
\begin{multline}\label{eq:Euler-product-presentation}
    \sumXiEta =
     \prod _{(\pi )} \sumXiEtaPi ,
     \quad \text{where }\sumXiEtaPi:=
    \\
    \sum _{\begin{subarray}{c}
    (\pidealb _i,\pidealc _i )_i \\
    \te{$\pi $-ideals,}
    \\
    \te{square-free}
        \end{subarray}}
    \epp
    \paren{\prod _{i=1}^m \mu (\pidealb _i)\mu (\pidealc _i)
    \Nrm (\pidealb _i)^{\frac 1{\log R}(-1+\xi _i\kyo )}
    \Nrm (\pidealc _i)^{\frac 1{\log R}(-1+\eta _i\kyo )}
    } 
.\end{multline}
Note that there are at most $\therank $ prime ideals of $\Ded $ over a given $(\pi ) $ and hence
there are at most $2^{\therank }$ square-free $\pi $-ideals.

By Lemma \ref{lem:pi-local-computation}, if $\Nrm (\pi \ded )\le w$ then $\sumXiEtaPi =1$; we also know the following when $\Nrm (\pi\ded )>w$, supposing (as we shall always do) that $w$ is large enough to invoke the lemma:
\begin{itemize}
    \item Unless $(\pidealb _i,\pidealc _i)=(\Ded ,\Ded )$,
    $(\Ded ,\idealp  )$, $(\idealp ,\Ded ) $ or $(\idealp ,\idealp )$ for some prime $\pi $-ideal $\idealp $ with $\Nrm (\idealp )=\Nrm (\pi\ded )$,
    we have $(0\le )\ \epp  \le 1/ \Nrm (\pi\ded )^2$;
    \item For those exceptional cases in the previous item, we know $\epp = 1$ for the first case and $=1/\Nrm (\pi\ded )$ for the others.
\end{itemize}
This gives us the following crude estimate, where $O_{\therank }(1)$ can be taken to be $2^{\therank }$:
\begin{multline}\label{eq:pi-crude-estimate}
    \sumXiEtaPi
    =
    \\
    1+\sum _{\idealp , \deg =1}
    \paren{
        -\Nrm (\pi\ded ) ^{-1+\frac 1{\log R}(-1+\xi _i\kyo )}
        -\Nrm (\pi\ded ) ^{-1+\frac 1{\log R}(-1+\eta _i\kyo )}
        +\Nrm (\pi\ded ) ^{-1+\frac 1{\log R}(-2+(\xi _i+\eta _i)\kyo )}
        } 
        \\
        + O_{\therank }\paren{ 1} \cdot \frac{1}{\Nrm (\pi\ded )^2}
.\end{multline}
In particular we have $\sumXiEta = 1+O_{\therank }\paren{\Nrm (\pi\ded )^{-1-\frac 1{\log R}}} $ uniformly in $\undl\xi ,\undl\eta $.
Therefore by basic facts on Euler products (such as \cite[Lemma 6.19]{KMMSY})
we conclude that the product $\prod _{(\pi )} \sumXiEtaPi $ converges absolutely.
As a result, the sum of absolute values associated with the sum $\sumXiEta$ \eqref{eq:consider-the-infinite-sum}
can be estimated as:
\begin{multline}\label{eq:the-sum-of-absolute-values}
    \sum _{(\idealb _i,\idealc _i)_i }
    \myep \cdot
    \paren{
        \prod _{i=1}^m
    \Nrm (\idealb _i)^{\frac 1{\log R}}
    \Nrm (\idealc _i)^{\frac 1{\log R}}
    }
    \\
    = \prod _{(\pi )} \paren{ 1+O_{\therank }\paren{\Nrmpi ^{-1-\frac 1{\log R}}}  }
    =  (\log R +O(1)) ^{O_r(1)},
\end{multline}
proving the convergence claim in Proposition \ref{prop:convergence-and-estimate}.

Now let us move on to the comparison of \eqref{eq:expression-after-Fourier} and \eqref{eq:convergence-and-estimate}.

\subsubsection{Summation and integration}

We want to replace the domain $\bbR $ of integration in \eqref{eq:expression-after-Fourier} by the bounded interval $I=[-\sqrt{\log R},\sqrt{\log R}]$.
\begin{lemma}\label{lem:change-of-domain-of-integration}
    Regarding the expression \eqref{eq:expression-after-Fourier},
    \begin{enumerate}
        \item 
    We have the following estimate:
    \begin{equation}
        \int _{\bbR } \Nrm (\idealb )^{\frac{1}{\log R} (-1+\xi \kyo ) } \chihat (\xi )d\xi
        =\int _{I } \Nrm (\idealb )^{\frac{1}{\log R} (-1+\xi \kyo ) } \chihat (\xi )d\xi
        \pm O_{A,\chi }\bigpa{\Nrm (\idealb )^{-\frac 1{\log R}} (\log R)^{-A} }
    \end{equation}
    for all ideals $\idealb \subset \Ded $. 
        \item Let $\undl{d\xi}\, \undl{d\eta }$ be a shorthand symbol for $ d\xi _1\dots d\xi _m d\eta _1 \dots d\eta _m $.
        For each $\idealtuple \in (\Ideals )^{2m}$ and positive number $A>0$, we have:
        \begin{align}\label{eq:expression-sum-int-interchanged}
            \prod _{i=1}^m
    &\paren{\chihatint{\idealb _i}{\xi _i} }
    \cdot \paren{\chihatint{\idealc _i}{\eta _i}}
    \\
          =  &\int _{I ^{2m}}  \Bigpa{\prod _i
            \Nrm (\idealb _i)^{\frac{-1}{\log R} (1-\xi _i \kyo )}
            \Nrm (\idealc _i)^{\frac{-1}{\log R} (1-\eta _i \kyo )}
            \cdot\chihat (\xi _i ) \chihat (\eta _i)
            } 
            \undl{d\xi}\, \undl{d\eta }
            \\
            &\pm O_{A,\chi ,m }\paren{
                (\log R)^{-A}\prod _{i=1}^m
                \Nrm (\idealb _i)^{\frac{-1}{\log R}}
                \Nrm (\idealc _i)^{\frac{-1}{\log R}}
            }
            .
        \end{align}
    \end{enumerate}
\end{lemma}
\begin{proof}
    (1)
    The error $\int _{\bbR \setminus I} \Nrm (\idealb )^{\frac 1{\log R}(-1+\xi\kyo )} \chihat (\xi ) d\xi $ is bounded in magnitute by
    \begin{equation}
        \Nrm (\idealb )^{\frac{-1}{\log R} } \int _{\bbR \setminus I} |\chihat (\xi )| d\xi
    \end{equation}
    which has at most the claimed size by Lemma \ref{lem:Fourier-textbook}.

    (2)
    Apply (1) to each of $\idealb _i$ and $\idealc _i$ and take the product,
    taking into account the bound
    \begin{equation}
        \int _I \Nrm (\idealb )^{\frac 1{\log R}(-1+\xi\kyo )  } \chihat (\xi ) d\xi
        = \Nrm (\idealb )^{\frac{-1}{\log R}} O_{\chi }(1) .
    \end{equation}
    This completes the proof.
\end{proof}

Apply the operation $(\log R)^{2m}\sum _{\idealtuple  } \prod _{i=1}^m \mu (\idealb _i)\mu (\idealc _i) \Bigpa{-} \myep $
to the estimate of
Lemma \ref{lem:change-of-domain-of-integration}(2).
The left hand side becomes precisely \eqref{eq:expression-after-Fourier}.
The main term of the right hand side becomes
\begin{multline}\label{eq:right-hand-side-becomes}
    (\log R)^{2m}\sum _{\idealtuple }\int _{I ^{2m}} F(\idealtuple ,\underline\xi ,\underline\eta ) d\underline\xi d\underline\eta ,
    \quad\te{ where }
    \\
    F(\idealtuple , \undl\xi ,\undl\eta ):=
    \paren{\prod _{i=1}^m
    \mu (\idealb _i)\mu (\idealc _i)\Nrm (\idealb _i)^{\frac{-1}{\log R}(1-\xi _i \kyo )}
    \Nrm (\idealc _i)^{\frac{-1}{\log R}(1-\eta _i \kyo )}
    \chihat (\xi _i)
    \chihat (\eta _i)
    }
    \myep
.\end{multline}
We claim that we can interchange the sum and integral here.
Indeed, by the convergence part of Proposition \ref{prop:convergence-and-estimate}, i.e., formula \eqref{eq:the-sum-of-absolute-values}, the sum $\sum _{\idealtuple }F(\idealtuple ,\undl\xi ,\undl\eta ) $
converges absolutely and uniformly in $\undl\xi $ and $\undl\eta $ to a continuous function.
Since $I$ is a bounded closed interval, our claim follows so that the value \eqref{eq:right-hand-side-becomes}
equals \eqref{eq:convergence-and-estimate}.

The error term is at most the following, which we can bound again by \eqref{eq:the-sum-of-absolute-values}:
\begin{equation}
    O_{A,\chi ,m} ((\log R)^{2m-A} )\sum _{\idealtuple  }
    \paren{
        \prod _{i=1}^m
        \Nrm (\idealb _i)^{\frac{1}{\log R}}
        \Nrm (\idealc _i)^{\frac{1}{\log R}}
    }
    \myep
    =
    O_{A,\chi ,m} ((\log R)^{2m-A +O_r(1)} )
.\end{equation}
The proof of Proposition \ref{prop:convergence-and-estimate} is now complete.

\subsection{Intermission}

Before proceeding further, let us recall basic facts from elementary calculus and the theory of the zeta function.
The absolute constans $c_i$ and $C_i$ appearing below can be made explicit, but we do not seek to do so because their precise values are not important.
Potentially big constants are denoted in upper case and
potentially small ones in lower case.
From \S \ref{sec:Euler} on, when we say some quantities should be small or large enough, we will be implicitly using these constants to specify the thresholds.

\subsubsection{Some calculus}

There is a positive real number $c_1>0$ such that for all $\varepsilon \in \bbC $ with $|\varepsilon |\le c_1$ one has
\begin{align}
    \label{eq:exp}
    e^{\varepsilon } &=1+O(1)\cdot\varepsilon \quad\te{ and }
    \\
    \label{eq:log}
    \log (1+\varepsilon )&=1+O(1)\cdot\varepsilon
\end{align}
with both $O(1)\le 2$. (Actually one can take $c_1:=1/2$, say.)
Next, for real numbers $A\ge 2$ we have Taylor expansion
\begin{equation}
    A^{-1+\varepsilon } = A^{-1}+\frac{\log A}{A}\varepsilon
    +\frac{1}{2}\frac{\log (A)^2}{A}\varepsilon ^2 +\cdots .
\end{equation}
Noting that the positively valued function $A\mapsto \frac{\log A}{A}$ has bounded range, there are $c_2>0$
and $C_3>0$ such that for all $A\ge 2$ and $\varepsilon \in \bbC $ with $|\varepsilon | \le \frac{c_2}{\log A}$ we have
\begin{equation}
    1-A^{-1+\varepsilon } = 1-A^{-1}+O(1)\cdot\frac{\log A}{A}\varepsilon
\end{equation}
with $O(1)\le C_3$. It follows that
\begin{equation}\label{eq:power-of-A}
    \frac{1-A^{-1+\varepsilon }}{1-A^{-1}}
    = 1 + \frac{O(1)}{1-A^{-1}}
    \frac{\log A}{A}\varepsilon
    = 1+ O(1)\frac{\log A}{A}\varepsilon
\end{equation}
with a new $O(1)$ constant $\le 2 C_3$.
Also, since the function $A\mapsto \frac{\log A}{A}$ decreases for $A\ge e$, for prime ideals $\idealp \subset \Ded $ with $\idealp \cap \ded = (\pi )$ we have
\begin{equation}\label{eq:log-N-over-N}
    \frac{\log \Nrm (\idealp )}{\Nrm (\idealp )}
    \le
    \frac{\log \Nrm (\pi \ded )}{\Nrm (\pi \ded )}
\end{equation}
at least if $\Nrm (\pi \ded ) \ge 3$.
This inequality happens to be true even when $\Nrm (\pi\ded )=2$ thanks to the equality $\frac{\log 4}{4}=\frac{\log 2}{2}$.

\subsubsection{The zeta function}\label{sec:zeta}
We need to recall the zeta function of $\Ded $.
For $s\in \bbC $ with $\Real (s)>1$, we set:
\begin{equation}\label{eq:def-of-zeta}
    \zeta _{\Ded } (s) := \prod _{\idealp \in \spDed } \paren{1-\frac{1}{\Nrm (\idealp )^s}}\inv
    \quad\paren{=\sum _{\ideala \in \Ideals } \frac{1}{\Nrm (\ideala )^s} }
.\end{equation}
It is known that $\zeta _{\Ded }(s)$ extends to a meromorphic function on $\bbC $.
Actually we know \cite[Theorem 5.9, p.53]{Rosen} that there is a polynomial
$L(u)\in \bbZ [u]$ of degree $2g$, with $L(q\inv )=\mgn{\Pic ^0(\kurve )}/q^g$, such that:
\begin{equation}
    \zeta _{\Ded }(s)
    \cdot
    \prod _{v\in \complm }
    \paren{1-\frac{1}{\mgn{\bbF (v)}^s } } \inv
    =
    \frac{L(q^{-s})}{(1-q^{-s}) (1-q^{1-s} )}
.\end{equation}
It follows that $\zeta _{\Ded } (s)$ has a simple pole at $s=1$ with positive residue, say $\kappa _{\Ded }> 0$:
\begin{equation}\label{eq:pole-of-zeta}
    \zeta _{\Ded } (s) = \frac{\kappa _{\Ded }}{s-1} + O_{\Ded } (1) \quad\te{ as }s\to 1.
\end{equation}
Moreover, by Weil's Riemann Hypothesis for algebraic curves \cite[Theorem 5.10, p.55]{Rosen}, we know that
all the roots of $L(u)$ in $\bbC $ have magnitude $q^{-1/2}$.
\Cheb 's density theorem \ref{thm:Cheb}
below is an important consequence of this.

By explicit computation we know $\zeta _{\Fq [t]} (s)=1/\paren{1-{q^{1-s}} } $ (see \cite[p.11]{Rosen}).
From this it follows for integers $i\ge 1$:
\begin{equation}\label{eq:Cheb-for-polynomial-ring}
    \mgn{ \br{\te{maximal ideals }\pi\ded \subset \ded \mid \Nrmpi = q^i } }
    = \frac{q^i}{i} + O\paren{ \frac{q^{i/2}}{i} },
\end{equation}
see \cite[Theorem 2.2, p.14]{Rosen}.
This
gives the following.
The sums are over maximal ideals $\pi\ded \subset \ded $ satisfying the indicated conditions:
\begin{align}
    \label{eq:large-pi}
    \sum _{\Nrmpi > w} \frac{1}{\Nrmpi ^2}
    &\le C_4 \frac{1}{w\log _q w} \quad\te{ and }
    \\
    \label{eq:small-pi}
    \sum _{\Nrmpi \le w} \frac{\log _q \Nrmpi }{\Nrmpi }
    &\le C_5 \log _q w
\end{align}
for some $C_4,C_5>0$ and all $w>1$.

\subsection{Euler product}\label{sec:Euler}
Now we compute the main term \eqref{eq:convergence-and-estimate} using the Euler product presentation \eqref{eq:Euler-product-presentation} and estimate \eqref{eq:pi-crude-estimate}.

We start with some detailed estimate of the Euler product.
Recall that the latter estimate requires $\Nrmpi >w$ and that $w$ be large enough.
Assume $w$ is large enough to match this requirement.
Then by \eqref{eq:pi-crude-estimate} and basic facts like
$(1+\varepsilon _1)(1+\varepsilon _2)(1+\varepsilon _1+\varepsilon _2)^{-1}=
1 +O(\varepsilon _1\varepsilon _2)$, we have for $\Nrmpi > w$:
\begin{multline}\label{eq:one-step-away-from-Euler}
    \sumXiEtaPi
    \\
    =
    \frac{
        \paren{1-\Nrmpi ^{-1+\frac{1}{\log R}(-1+\xi_i\kyo )  } }
        \paren{1-\Nrmpi ^{-1+\frac{1}{\log R}(-1+\eta _i\kyo )  } }
    }
    {
        1-\Nrmpi ^{-2+\frac{1}{\log R}(-2+(\xi_i+\eta _i)\kyo )  }
    } 
    \cdot (1+O_{\therank } \paren{\Nrmpi ^{-2} } )
.\end{multline}
Take the product of \eqref{eq:one-step-away-from-Euler}
over all $\pi\ded $ with $\Nrmpi >w$.
By the definition of the zeta function $\zeta _{\Ded }$ in \eqref{eq:def-of-zeta}, we get the following.
There, the symbol $\prod _{\idealp \te{ with } \Nrmpi \le w} $ means the product over prime ideals $\idealp $ of $\Ded $ such that $\pi\ded :=\idealp \cap \ded $ satisfies the indicated condition:
\begin{multline}\label{eq:by-very-def-of-zeta}
    \sumXiEta = \\
    \prod _{i=1}^m
    \left(
    \frac
    {
        \prod\limits _{\idealp \te{ with } \Nrmpi \le w}
        \paren{1-\Nrm (\idealp )^{-1+\frac{1}{\log R}( -1+\xi _i\kyo  )  }
        }
    }
    {
                            \zeta _\Ded \paren{1+\frac{1}{\log R}(1-\xi_i\kyo )}
    }
    \frac
    {\prod\limits _{\idealp \te{ with } \Nrmpi \le w}
        \paren{1-\Nrm (\idealp )^{-1+\frac{1}{\log R}( -1+\eta _i\kyo  )  }
        }
    }
    {\zeta _\Ded \paren{1+\frac{1}{\log R}(1-\eta _i\kyo )}
    }
    \right.
    \\[5mm]
    \cdot
    \left.
    \frac
    {\zeta _\Ded \paren{1+\frac{1}{\log R}(2-(\xi_i+\eta _i)\kyo ) }
    }
    {\prod\limits _{\idealp \te{ with } \Nrmpi \le w}
        \paren{1-\Nrm (\idealp )^{-1+\frac{1}{\log R}( -2+(\xi _i+\eta _i)\kyo  )  }
        }
    }
    \right)
    \cdot (1+mO_{\therank } (1/w\log _q w ) ) .
\end{multline}
Here the last factor has been obtained using \eqref{eq:exp}, \eqref{eq:log} via $\exp \circ \log = \id $ and \eqref{eq:large-pi} as follows:
\begin{align}
    \prod _{\Nrmpi>w} (1+O_r\paren{ \Nrmpi^{-2}})
    &=
    \exp \paren{ \sum _{\Nrmpi >w} O_r\paren{ \Nrmpi ^{-2} } } \\
    &=\exp (O_r \paren{1/(w\log _q w)})
    =1+O_r(1/(w\log _q w))
,\end{align}
where for the last estimate we have to assume $O_r(1)/(w\log _q w)$ is small enough.
Since we have $m$ such factors, we get the factor $1+mO_r(1/w\log _q w)$.

Formula \eqref{eq:pole-of-zeta} can be written as
$\zeta _{\Ded }(1+\varepsilon ) = \frac{\kappa _{\Ded }}{\varepsilon }(1+O_{\Ded }(\varepsilon )) $ with $\varepsilon \in \bb C$ close to $0$.
Applying this to $\varepsilon =\frac{1}{\log R} (1-\xi _i \kyo ) $,
$\frac{1}{\log R} (1-\eta _i \kyo ) $ and $\frac{1}{\log R} (2-(\xi _i+\eta _i )\kyo )$
all of size $O(\frac{1}{\sqrt{\log R}})$,
we find that the product of zeta functions in \eqref{eq:by-very-def-of-zeta} has the following form when $R$ is large enough:
\begin{equation}
    \paren{\frac{1}{\kappa _{\Ded}\log R} }^m
    \cdot
    (1+mO_{\Ded}\paren{\frac{1}{\sqrt{\log R}} } )
    \cdot
    \prod _{i=1}^m\frac{(1-\xi _i\kyo )(1-\eta _i\kyo )}{2-(\xi _i+\eta _i )\kyo }
.\end{equation}
We have to compute the products $\prod _{\idealp \te{ with }\Nrmpi \le w}$ in \eqref{eq:by-very-def-of-zeta} as well.
By \eqref{eq:power-of-A} we know for small complex numbers $\varepsilon $:
\begin{equation}\label{eq:product-small-idealp}
    \prod _{\idealp \te{ with }\Nrmpi \le w}
    \paren{1-\Nrm (\idealp )^{-1+\varepsilon }}
    =
    \prod _{\idealp \te{ with }\Nrmpi \le w}
    \paren{ 1-\Nrm (\idealp )^{-1} }
    \paren{ 1+O(1)\frac{\log \Nrm (\idealp )}{\Nrm (\idealp )}\varepsilon }
.\end{equation}
The product of the first factors is $\frac{\totient (W)}{\Nrm (W\Ded )}$.
For the second factors, by \eqref{eq:log-N-over-N},
\eqref{eq:small-pi} and the fact that the number of prime ideals
$\idealp $ over a given $\pi \ded $ is at most $\therank $, for small $\varepsilon $ we have:
\begin{align}
    (\te{product of the second factors in \eqref{eq:product-small-idealp}})
    &=\prod _{(\pi ) \te{ with }\Nrm (\pi )\le w}
    \paren{1+O(1)\frac{\log \Nrm (\pi\ded )}{\Nrmpi }\varepsilon }^r
    \\
    &= 1+O(1)\therank \varepsilon \log _q w .
\end{align}
We apply this to $\varepsilon = \frac{1}{\log R}(-1+\xi \kyo )$, with $\xi =\xi _i $ or $\eta _i$
and to $\frac{1}{\log R}(-2+(\xi _i+\eta _i)\kyo )$
which are of size $O(\frac 1{\sqrt{\log R}})$.
It follows that if $\frac 1{\sqrt{\log R} }$ is smaller than $\frac{mr}{\log _q w}$ times an absolute constant, then the product of the products $\prod _{\idealp \te{ with }\Nrmpi \le w}$ in \eqref{eq:by-very-def-of-zeta} is of the form:
\begin{equation}\label{eq:product-of-products}
    \paren{\frac{\Nrm (W\Ded )}{\totient (W)}}^m
    \paren{1+3mrO(1)\frac{\log _q w}{\sqrt{\log R}} }
,\end{equation}
with $O(1)$ an absolute constant.
By \eqref{eq:by-very-def-of-zeta}--\eqref{eq:product-of-products}, we get an estimate:
\begin{multline}\label{eq:essentially-the-conclusion}
    \exieta = (1+mO_{\Ded }\paren{\frac{1}{\sqrt {\log R}}} )(1+mO_{\therank }\paren{\frac{1}{w\log w}}) (1+3m\therank O\paren{\frac{\log _q w}{\sqrt{\log R} } } )
    \\
    \cdot \paren{\frac{1}{\kappa _{\Ded}\log R } \frac{\Nrm (W\Ded )}{\totient (W)} }^m
    \cdot \prod_{i=1}^m\frac
    { \paren{1-\xi_i\kyo }
    \paren{1-\eta_i\kyo }
    }
    { 2-(\xi_i+\eta_i)\kyo
    }
.\end{multline}
The error factor above is a
$1+O_{m,\therank }\paren{\frac 1{w\log _q w} } + O_{\Ded ,\therank ,m}\paren{\frac{\log _q w}{\sqrt{\log R}}}$.

We are ready to compute \eqref{eq:convergence-and-estimate}.
Let us recall for the convenience of reference:
\[
    \eqref{eq:convergence-and-estimate}
    :=
    (\log R)^{2m}\int _{I^{2m}} \sumXiEta \paren{\prod _{i=1}^m \chihat (\xi _i )\chihat (\eta _i) }d\undl\xi d\undl\eta .
\]
\begin{proposition}\label{prop:one-step-away-from-th!}
    We have
    \begin{multline}\label{eq:one-step-away-from-th!}
        (\log R)^{2m}\int _{I^{2m}} \sumXiEta \paren{\prod _{i=1}^m \chihat (\xi _i )\chihat (\eta _i) }d\undl\xi d\undl\eta
    = \\
    \paren{
        C_\chi \frac{\log R}{\kappa _{\Ded }}\frac{\Nrm (W\Ded )}{\totient (W)}
        }^m
        \paren{
            1
            +O_{m,\therank ,\chi }
            \paren{
                \frac 1{w\log _q w}
            }
            +O_{\Ded , m,\therank ,\chi }
            \paren{
                \frac {w\log _q w}{\sqrt{\log R}}
            }
        }
        .
    \end{multline}
\end{proposition}
\begin{proof}
By substituting \eqref{eq:essentially-the-conclusion} we get
\begin{multline}\label{eq:by-substituting}
\text{L.H.S.\ of }\eqref{eq:one-step-away-from-th!}
=
\paren{
    \frac{\log R}{\kappa _{\Ded } }\frac{\Nrm (W\Ded )}{\totient (W)}
}^{m}
\paren{
\int _{I^{2m}}
\prod _{i=1}^m
\frac{(1-\xi _i\kyo)(1-\eta _i \kyo )}{2-(\xi _i+\eta _i )\kyo }
\chihat (\xi _i )\chihat (\eta _i)
d\undl{\xi }d\undl{\eta }
}
\\
\cdot \paren{
    1+O_{m,r}\paren{
        \frac{1}{w\log _q w}
        }
    +O_{\Ded ,r,m}\paren{
        \frac{\log _q w }{\sqrt{\log R}}
        }
} 
\end{multline}
Write $F(\xi _i,\eta _i):= \frac{(1-\xi _i\kyo)(1-\eta _i \kyo )}{2-(\xi _i+\eta _i )\kyo }
\chihat (\xi _i )\chihat (\eta _i) $ and
$\undl F (\undl \xi ,\undl \eta ):= \prod_{i=1}^m F(\xi _i,\eta _i) $ for short.
The value above is estimated as:
\begin{align}\label{eq:formerly-I-was-fool}
    =
    \paren{
        \frac{\log R}{\kappa _{\Ded } }\frac{\Nrm (W\Ded )}{\totient (W)}
    }^{m}
    &\left[
        \int _{I^{2m}} \undl F(\undl \xi ,\undl \eta ) d\undl{\xi }d\undl{\eta }
    \right.
        \\
        &
        \left.
        +
        \int _{I^{2m}}
        \mgn{
            \undl F(\undl \xi ,\undl \eta )
            }
            d\undlxi d\undl{\eta }
            \paren{
                O_{m,r}\paren{
                    \frac{1}{w\log _q w}
                }
                +O_{\Ded ,r,m}
                \paren{
                    \frac{\log _q w }{\sqrt{\log R}}
                }
            }
    \right]
    .
\end{align}
We want to bound the integral $\int _{I^{2m}} |\undlF (\undlxi ,\undleta )| d\undlxi d\undleta $.
By Lemma \ref{lem:Fourier-textbook}, we know for any $A>0$
\[
    \chihat (\xi _i)=O_{\chi,A}((1+|\xi _i|)^{-A} )
\]
and therefore for any $A>0$
\begin{equation}\label{eq:bound-of-F}
    \mgn{
        F(\xi _i,\eta _i)
    }
    \le
    \frac 12 O_{\chi,A}( (1+|\xi _i |)^{-A}(1+|\eta _i |)^{-A} ).
\end{equation}
It follows that (by taking $A:= 2$ for example)
\begin{align}\label{eq:estimate-of-int-F-i}
    \int _{I^2} | F(\xi _i,\eta _i) | d\xi _i d\eta _i
    &= O_\chi (1)\int _{\mathbb R^2} | F(\xi _i,\eta _i) | d\xi _i d\eta _i
    \\
    & =O_\chi (1)\int _{\mathbb R^2} | (1+|\xi _i|)^{-2}(1+|\eta _i|)^{-2} | d\xi _i d\eta _i
    \\
    &=O_\chi (1)\paren{\int _{\mathbb R} | (1+|\xi |)^{-2} | d\xi _i d\eta _i }^2
\end{align}
is a finite value. Hence
$\int _{I^{2m}} |\undlF (\undlxi ,\undleta )| d\undlxi d\undleta
= \prod _{i=1}^m\int _{I} |F(\xi _i ,\eta _i) | d\xi _i d\eta _i
= O_{\chi }(1)
$ is also a finite value.

Next we consider the integral $\int _{I^{2m}} \undlF (\undlxi ,\undleta )d\undlxi d\undleta $.
We want to replace $\int _{I^{2m}}$ by $\int _{\mathbb R^{2m}}$.
Consider the following partition of the domain of integral:
\[ \mathbb R^{2m} = (I \sqcup (\mathbb R \setminus I ) )^{2m}
 = I^{2m } \sqcup \bigsqcup _{J } \Omega _{J } \]
 where $J $ runs through maps $\br{1,\dots ,2m} \to \br{I , \mathbb R \setminus I}$
 except the constant map into the one-point set $\br{I}$,
 and $\Omega _J$ denotes the corresponding product $\Omega _J := J_1 \times \dots \times J_{2m} \subset \mathbb R^{2m}$.

 \begin{lemma}\label{lemma:estimate-int-F}
     For any $J$ as above and any $A>0$, we have the estimate
     $
     \mgn{
         \int _{\Omega _J} \undlF (\undlxi ,\undleta )d\undlxi d\undleta
     }
     \le
     O_{\chi,A } ( (\log R)^{-A} )
     $.
 \end{lemma}

 \begin{proof}[Proof of Lemma \ref{lemma:estimate-int-F}]
    Since $J$ is not the constant function at $\br{I}$, there is an index $1\le k\le 2m$ such that $J_k = \bbR \setminus I$.
    By symmetry, we may assume $k=1$.
    By \eqref{eq:bound-of-F} \eqref{eq:estimate-of-int-F-i}, we have for any $A>2$
    \begin{align}
        \mgn{ \int _{\Omega _J} \undlF (\undlxi ,\undleta )d\undlxi d\undleta }
        &\le
        \int _{\Omega _J} d\undlxi d\undleta
        \left[
            O_{\chi,A}\paren{(1+|\xi _i|)^{-A}(1+|\eta _i|)^{-2}}\prod _{i=2}^m O_{\chi}\paren{ (1+|\xi _i|)^{-2}(1+|\eta _i|)^{-2} }
        \right]
        \\
        &=
        \paren{
            \int _{\bbR \setminus I} O_{\chi,A}\paren{ (1+|\xi _1|)^{-A}} d\xi
        }
        \cdot O_{\chi }(1)
        \\
        &= O_{\chi,A}\paren{
            \paren{\sqrt{\log R} }^{-A+1} }.
    \end{align}

    This completes the proof of Lemma \ref{lemma:estimate-int-F}.
 \end{proof}
By Lemma \ref{lemma:estimate-int-F}, we can proceed as:
\begin{align}\label{eq:we-can-proceed-as}
    \int _{\bbR ^{2m}} \undlF (\undlxi ,\undleta ) d\undlxi d\undleta
    &= \paren{\int _{I^{2m}} + \sum _{J} \int _{\Omega _J} }\undlF (\undlxi ,\undleta ) d\undlxi d\undleta
    \\
    &= \int _{I^{2m}}\undlF (\undlxi ,\undleta ) d\undlxi d\undleta
    + (2^{2m}-1) O_{\chi, A}\paren{ \paren{\log R}^{-A} }
    \\
    &= \int _{I^{2m}}\undlF (\undlxi ,\undleta ) d\undlxi d\undleta
    + O_{\chi, A,m}\paren{ \paren{\log R}^{-A} }
\end{align}
for any $A>0$.

Now that we have estimated the integrals in \eqref{eq:formerly-I-was-fool}
in \eqref{eq:estimate-of-int-F-i} and \eqref{eq:we-can-proceed-as},
we obtain for any $A>0$:
\begin{align}
    \text{L.H.S.\ of }\eqref{eq:one-step-away-from-th!}
    =
    \paren{\frac{\log R}{\kappa _\Ded }
        \frac{\Nrm (W\Ded )}{\totient (W)}
    }^m
    &
    \left[
        \int _{\mathbb R^{2m}} \undlF (\undlxi ,\undleta ) d\undlxi d\undleta
        +
        O_{\chi,A,m}\paren{ (\log R)^{-A}}
    \right.
        \\
        &
        \left.
            +
            O_{m,r,\chi }\paren{
                \frac 1{w\log _q w}
            }
            +O_{\Ded ,r,m,\chi }\paren{\frac{\log _q w}{\sqrt{\log R}}}
    \right]
    .
\end{align}
If we set $A=1/2$, the term $O_{\chi,A,m}\paren{ (\log R)^{-A}}$ can be absorbed into
$O_{\Ded ,r,m,\chi }\paren{\frac{\log _q w}{\sqrt{\log R}}}$.
The main term
$\int _{\mathbb R^{2m}} \undlF (\undlxi ,\undleta ) d\undlxi d\undleta
= \paren{\int _{\mathbb R^2} F(\xi ,\eta  ) d\xi d\eta }^m$
can be evaluated by
a standard Fourier analysis computation (e.g.\ \cite[p.170]{Tao} or \cite[Lemma 6.29]{KMMSY}): 
\begin{equation}\label{eq:def-of-C-chi}
    \int _{\bbR ^2} \frac{(1-\xi _i\kyo )(1-\eta _i\kyo )}{2-(\xi _i+\eta _i)\kyo }
    \chihat (\xi _i)\chihat (\eta _i) d\xi _i d\eta _i
    =
    \int _0 ^{+\infty } \chi '(x)^2 dx =: C_\chi .
\end{equation}
We conclude that
\[
    \text{L.H.S.\ of }\eqref{eq:one-step-away-from-th!}
    =
    \paren{\frac{\log R}{\kappa _\Ded }
        \frac{\Nrm (W\Ded )}{\totient (W)}
    }^m
    \left[
        (C_\chi )^{m}
            +
            O_{m,r,\chi }\paren{
                \frac 1{w\log _q w}
            }
            +O_{\Ded ,r,m,\chi }\paren{\frac{\log _q w}{\sqrt{\log R}}}
    \right]
    .
\]
This completes the proof of Proposition \ref{prop:one-step-away-from-th!}
\end{proof}

Let us collect the computations we have done and
finish the proof of the main result of this section.
\begin{proof}[Proof of Theorem \ref{thm:Goldston-Yildirim}]

We wanted to evaluate up to error the average:
\[ \Ex \left( \prod _{i=1}^m \Manga (W\phi _i (x)+b_i)^2  \emidd
x\in B
\right)
.
\]
By \eqref{eq:unfolded-expression} and \eqref{eq:expression-after-Fourier}, it equals:
\begin{multline}
    =\sum _{(\idealb _i ,\idealc _i )_i \in \Ideals ^{2m}} (\log R)^{2m}
    \left(
    \prod _{i=1}^m
    \mu (\idealb _i)\mu (\idealc _i)
    \chihatint{\idealb _i}{\xi _i}
    \right.
    \\
    \left.
    \cdot \chihatint{\idealc _i}{\eta _i}
    \right)
    \cdot \myep
    .
\end{multline}
By Proposition \ref{prop:convergence-and-estimate}, this has been estimated as
\begin{equation}
    =(\log R)^{2m}\int _{I^{2m}} \sumXiEta \paren{\prod _{i=1}^m \chihat (\xi _i )\chihat (\eta _i) }d\undl\xi d\undl\eta
    +
    O_{A,\chi ,m,\therank } ((\log R)^{-A})
\end{equation}
for any $A>0$.
By Proposition \ref{prop:one-step-away-from-th!},
this is further estimated as:
\[     =\paren{
    C_\chi \frac{\log R}{\kappa _{\Ded }}\frac{\Nrm (W\Ded )}{\totient (W)}
    }^m
    \paren{
        1
        +O_{m,\therank ,\chi }
        \paren{
            \frac 1{w\log _q w}
        }
        +O_{\Ded , m,\therank ,\chi }
        \paren{
            \frac {w\log _q w}{\sqrt{\log R}}
        }
    }
    +
    O_{A,\chi ,m,\therank } ((\log R)^{-A})
    .
\]
By setting $A:=1$, the last error term can be absorbed in $O_{\Ded , m,\therank ,\chi }
\paren{
    \frac {w\log _q w}{\sqrt{\log R}}
}$ in the parentheses.
(Note that $\frac{\Nrm (W\Ded )}{\totient (W)}\le 1$ regardless of the specific $W$.)

The proof of Theorem \ref{thm:Goldston-Yildirim} is thus completed.
\end{proof}

\section{\Cheb\ and the end of proof}\label{sec:end-of-proof}

As always, let $\Ded $ continue to be a Dedekind domain finitely generated over $\Fp $.
We restate Theorem \ref{thm:density} in a slightly broader generality.
In the number field case \cite{KMMSY}, the extra generality allowed one to prove a constellation theorem for prime-valued points on a binary quadratic form $ax^2+bxy+cy^2$ over $\bbZ $.

\begin{definition}
    Let $\ideala \subset \Ded $ be a non-zero ideal.
    Let us define the set $\pele $ of {\em prime elements of $\ideala $}
    by
    \begin{equation}
        \pele := \br{ \alpha \in \ideala \mid \ideala /\alpha\Ded \cong \Ded /\idealp \te{ as an $\Ded $-module for some maximal ideal $\idealp $} }
    .\end{equation}
\end{definition}
Note that $\mcal P_{\Ded }$ is exactly the set of prime elements of $\Ded $ because an isomorphism
$\Ded /\alpha\Ded \cong \Ded /\idealp $ of $\Ded $-modules forces the equality $\alpha\Ded = \idealp $.

Now we can state our main theorem in its proper generality.

\begin{theorem}\label{thm:homothetic}
    Let $\Ded $ be a Dedekind domain finitely generated over $\Fp $
    and $\ded \subset \Ded $ as in Proposition \ref{prop:choice-of-subring}.
    Let $\ideala \subset \Ded $ be a non-zero ideal and $S\subset \ideala $ a finite subset of it.
    Then any subset $A\subset \pele  $ of positive relative density
    contains a non-trivial $\ded $-homothetic copy of $S$.
\end{theorem}

For the proof, we need to recall \Cheb 's density theorem.
This is by far the deepest imput from algebraic geometry in this work.
Let $\kurve $ be a complete non-singular geometrically irreducible curve over $\Fq $.
Let $\Pic (\kurve ) \xrightarrow {\phi }G$ be a finite quotient of $\Pic (\kurve )$.
The restriction of the degree map $\deg \colon \ker (\phi )\to \bbZ $ is necessarily non-trivial.
Let $D\ge 1 $ be the order of its cokernel.
(Let us always use $D$ in this sense when $G$ is understood.)
It follows that we have the degree map of the following form:
\begin{equation}
    \deg \colon G\to \bbZ / D\bbZ .
\end{equation}

The next result is a consequence of Weil's Riemann Hypothesis 
for algebraic curves over finite fields. 
\begin{theorem}[\Cheb 's density theorem]\label{thm:Cheb}
    Let $G$ be a finite quotient of $\Pic (\kurve )$
    and $P\in G$.
    Then for positive integers $n>0$, we have the following cardinality estimate:
    \begin{equation}
        \mgn{ \Bigl\{ x\in \kurve \mid \deg (x)=n\te{ in $\bbZ $ and }[x]=P \te{ in }G \Bigr\} }
        =
        \begin{cases}
            \displaystyle \frac{D}{\mgn G} \frac{q^n}n + O_{\kurve } \paren{\frac{q^{n/2}}n }
            & \te{ if } n\equiv \deg (P) \te{ in }\bbZ / D\bbZ ,
            \\[5mm]
            \displaystyle  O_{\kurve } \paren{\frac{q^{n/2}}n } & \te{ else. }
        \end{cases}
    \end{equation}
\end{theorem}
\begin{proof}
    A slightly weaker version of this statement can be found in \cite[Theorem 9.13B, p.125]{Rosen}.
    Our statement can be obtained by the same argument
    by using \cite[Proposition 9.21, p.137]{Rosen} and \cite[Theorem 9.24, p.141]{Rosen} in place of \cite[Theorem 9.16B, p.129]{Rosen}.
\end{proof}

Via the transition of parameters from the degree $n=\deg (x) $ to the norm $L=q^{n}=:\Nrm (x) $, 
we get the following:
\begin{corollary}\label{cor:conseq-of-Cheb}
    For all sufficiently large positive numbers $L>0$, we have the following estimate:
    \begin{equation}
            \frac{1}{q^D} \frac{D}{\mgn G} \frac{L}{\log _q L}
            \le \mgn{ \Bigl\{ x\in \kurve  \Bigm| \Nrm (x) \le L \te{ and }[x]+P=0 \te{ in }G \Bigr\} }
            \le 2\frac{D}{\mgn G} \frac{L}{\log _q L}
    .\end{equation}
\end{corollary}

The following is clear from definitions.
\begin{lemma}\label{lem:clear-from-definitions}
    If $\alpha \in \ideala $ is a prime element, then:
    \begin{equation}
        0\le \Lambda _{R,\chi }^{\ideala } (\alpha ) \le \log R .
    \end{equation}
    The right-hand `$\le $' is an equality if $\Nrm (\alpha )\ge \Nrm (\ideala ) R $.
\end{lemma}

When we apply Lemma \ref{lem:clear-from-definitions}, it will be convenient to have a bound for the number of
elements $\alpha \in \ideala $
with $\Nrm (\alpha )< \Nrm (\ideala ) R $.
For $L\ge 1$, let us write
\begin{equation}
    \ideala (L):= \br{\alpha \in \ideala \mid \Nrm (\alpha )\le L }.
\end{equation}

\begin{corollary}\label{cor:bound-for-elements-with-small-ideal-norm}
    For positive real numbers $L>1$ and $M>1$, we have the bound
    \begin{equation}
        \mgn{ \pele \cap \ideala (\Nrm (\ideala ) L)\cap \ideala _{\le M}}
        = O_{\Ded }\paren{(\log M)^{n-1}\frac L{\log _q L}}.
    \end{equation}
\end{corollary}
\begin{proof}
    Given $\alpha \in \pele \cap \ideala (\Nrm (\ideala ) L)$, the ideal $\idealp := \alpha \ideala\inv \subset \Ded $ is a prime ideal
    whose class in $\Pic (\Ded )$ equals $-[\ideala ]$ and norm equals $\Nrm (\alpha )/\Nrm (\ideala )$ ($\le L$).
    Two elements $\alpha ,\alpha '$ give the same $\idealp $ if and only if they are associate to each other.
    By Corollary \ref{cor:conseq-of-Cheb}, it follows that there are at most $2\frac{D}{\mgn{\Pic (\Ded )}}\frac{L}{\log _q L}$
    associate classes inside $\pele \cap \ideala (\Nrm (\ideala )L) \cap \ideala _{\le M}$.
    By Proposition \ref{prop:number-of-associates} applied to $\ideala _{\le M}\subset \Ded _{\le M}$, each associate class contains at most $O_{\Ded }((\log M)^{n-1})$ elements.
    This completes the proof.
\end{proof}

\subsection{Proof of the main result} 

Now we are ready to prove our main result.
\begin{proof}[Proof of Theorem \ref{thm:homothetic}]
    As always set $\spDed =\Spec \Ded $ and let $\kurve $ be its non-singular compactification.
    Let $\ded =\Fq [t] \subset \Ded $ be as in Proposition \ref{prop:choice-of-subring}.
    Recall from \eqref{eq:some-quantities} the definitions of some integer quantitites:
    \begin{equation}
        \begin{array}{cl}
            \integerparameters
        .\end{array}
    \end{equation}
    Also, let $\chi \colon \bbR \to [0,1]$ be a compactly supported $C^\infty $ function as in Definition \ref{def:von-Mang-function}.
%
    Take a norm-length compatible $\Ded ^*$-fundamental domain $\domain $ of $\ideala \nonzero $ which exists thanks to Proposition \ref{prop:norm-length-compatible-domain}.
    This means that there is a small positive number $c_{\domain }>0$ such that the following inclusion holds
    for all $M>0$:
    \begin{equation}\label{eq:norm-length-compatibility-again}
        \ideala  (c_{\domain }M^n) \cap \domain \subset \ideala _{\le M } .
    \end{equation}
    Let $\delta >0$ be any positive number smaller than the upper density $\bar\delta _{\pele }(A)$ of $A\subset \pele $.
    Let $\delta _1$ be the positive number defined by \eqref{eq:def-of-delta1} below (which is not very motivating) depending only on the preliminary data
    $\Ded $, $\ideala $, $\chi $, $S$, $\delta $ and $\domain $ that are already available.
    Using the relative \Sz\ theorem \ref{thm:Sz}, we fix the following positive numbers:
    \begin{equation}
        \rho := \rho _{}(\ded ,\ideala ,S, \delta _1),
        \quad
        \gamma := \gamma _{}(\ded ,\ideala ,S,\delta _1).
    \end{equation}

    Let $w>1$ be a large integer to be specified in a moment
    and $R>1$ be a large real number to be specified much later, satisfying
    \begin{equation}\label{eq:R-large-depending-on-w}
         R > w^{\therank }
    .\end{equation}
    Recall $k:=\mgn S$ and consider the maps
    in Definition \ref{def:S-N-rho-o-pseudorandom}:
    $\psi _{S,j}\circ \res _\omega \colon \ded ^k\oplus \ded ^k \surj \Ded $ for $1\le j\le k$ and $\omega \colon \br{1,\dots ,\hat j ,\dots ,k}\to \br\pm $.
    The number of indices $(j,\omega )$ is
    \begin{equation}
        m:= k2^{k-1}.
    \end{equation}
    It is routine to check that this family of maps
    satisfies the hypothesis of Theorem \ref{thm:Goldston-Yildirim}; see \cite[Lemma 5.8]{KMMSY} for details.
    Hence if $w$ is large enough depending on $S\subset \ideala $ and $\therank $, and if $R$ is large enough depending in addition on $\chi $ and $w$,
    then the error terms in Theorem \ref{thm:Goldston-Yildirim} can be made smaller than $\rho $:
    \begin{equation}\label{eq:conseq-GY}
        \mgn{ O_{m,r}\paren{\frac{1}{w\log _q w}} + O_{\chi,m,\therank ,\Ded }\paren{\frac{\log _q w}{\sqrt{\log R} }}
        }
        < \rho .
    \end{equation}
    We fix such $w $. The value of $R$ is yet to be fixed.

    Set
        $W:=\prod _{\Nrm (\pi\ded )\le w} \pi \in \ded $,
    where the product $\prod _{\Nrm (\pi\ded )\le w}$ is taken over the monic irreducible polynomials $\pi $ satisfying the indicated condition.

    Let $\param > 1$ be a large positive integer to be specified toward the end of the proof. 
    We consider the following positive real numbers determined by $\param $:
    \begin{equation}\label{eq:def-of-our-parameters}
        M= q^\param , \quad
        L= \frac{c_{\domain }M^n}{\Nrm (\ideala )}, \quad
        N= M/\lnorm W ,\quad
        R=N^{1/(2m+1)}.
        \end{equation}

    Since $\delta <\bar \delta _{\pele }(A)$ by our choice, 
    for infinitely many $\param \in \bbN $ the following inequality holds:
    \begin{equation}\label{eq:for-infinitely-many-d}
        \mgn {A \cap \ideala _{\le M} }
         > \delta \cdot \mgn{ \pele \cap \ideala _{\le M} }
    .\end{equation}
    By \eqref{eq:norm-length-compatibility-again}, 
    the set $\ideala _{\le M}$ contains $\ideala (\Nrm (\ideala ) L)\cap \domain $.
    Hence the right hand side is at least:
    \begin{equation}
        \ge
        \delta \cdot \mgn {\pele \cap \ideala (\Nrm (\ideala ) L 
        ) \cap \domain }
    .\end{equation}
    For every element $\alpha \in \pele \cap
        \ideala ( \Nrm (\ideala )L )
    \cap \domain $,
    the ideal $\alpha \ideala\inv \subset \Ded $ is a prime ideal
    with norm $\Nrm (\alpha )/\Nrm (\ideala )$ and whose class in $\Pic (\Ded )$ equals $-[\ideala ]$.
    Therefore the association $\alpha \mapsto \alpha\ideala\inv $ establishes a bijection from
    $\pele \cap 
        \ideala ( \Nrm (\ideala )L )
    \cap \domain $
    to the following set:
    \begin{equation}\label{eq:bounded-also-from-below}
     \Bigbr{ \idealp \in |\Spec (\Ded ) | \Bigm|
     \Nrm (\idealp )\le L  \te{ and }[\idealp ]+[\ideala ]=0\te{ in }\Pic (\Ded ) }
    .\end{equation}
    Its cardinality is already estimated in Corollary \ref{cor:conseq-of-Cheb}.
    As a result we get:
    \begin{equation}\label{eq:L/N(a)}
        \mgn{A\cap \ideala _{\le M}}> \delta\cdot\frac 1{q^D}\frac{D}{\mgn{ \Pic (\Ded )} }\frac{L}{\log _q L }
        =: \delta\cdot C_{\Ded }\frac{L}{\log _q L}
    ,\end{equation}
    where we have written $C_{\Ded }:= \frac 1{q^D}\frac{D}{\mgn{ \Pic (\Ded )} } $ for short.

    Since we want to use Lemma \ref{lem:clear-from-definitions} later, we want to consider only those elements with ideal norm $>\Nrm (\ideala ) R$.
    By Corollary \ref{cor:bound-for-elements-with-small-ideal-norm} we know
    \begin{equation}
        \mgn{\ideala (\Nrm (\ideala ) R )\cap \ideala _{\le M} } = O_{\Ded }\paren{(\log M)^{n-1}\frac{ R}{\log R} }.
    \end{equation}
    By \eqref{eq:def-of-our-parameters} the right hand side has the order of $(\log L)^{n-1} L^{1/(2m+1)n} $ or less as a function of $\param $,
    which is smaller than the right-most term of \eqref{eq:L/N(a)}.
    Hence by replacing $\delta $ by a slightly smaller value if necessary, we see that the following variant of \eqref{eq:L/N(a)} is valid:
    \begin{equation}\label{eq:L/N(a)-bis}
        \mgn{A\cap
        \paren{
            \ideala _{\le M}\setminus \ideala (\Nrm (\ideala ) R)
            }
        }
        >
        \delta\cdot C_{\Ded }\frac{L}{\log _q L}
    .\end{equation}
    \begin{lemma}\label{lem:trimming-small-alpha}
        For every $\alpha \in \pele \setminus \ideala (\Nrm (\ideala ) R )$, 
        the residue class of $\alpha $ in $\ideala /W\ideala $ generates it as an $\Ded $-module.
    \end{lemma}
        \begin{proof}[Proof of Lemma]
            By Chinese Remainer Theorem for $\Ded $-modules,
            the assertion is equivalent to that
            $\alpha \in \ideala \setminus
            (\bigcup
            _{\idealp |W}
            \idealp \ideala )$.
            Suppose there is a $\idealp |W $ such that $\alpha \in \idealp\ideala $.
        Since $\alpha \in \pele $ it follows that $\alpha =\ideala \idealp $.
        By the definition of $W$ the ideal $\pi\ded =\idealp \cap\ded $
        has norm $\le w $. It follows that $\Nrm (\idealp )\le w^\therank $
        and hence
        \begin{equation}\label{eq:which-cannot-happen}
            \Nrm (\ideala ) R < \Nrm (\alpha )\le \Nrm (\ideala ) w^r
            .
        \end{equation}
        This contradicts the assumption \eqref{eq:R-large-depending-on-w}.
        This proves Lemma \ref{lem:trimming-small-alpha}.
        \end{proof}

    As $\ideala $ is a rank 1 projective $\Ded $-module, we have an isomorphism of $\Ded $-modules $\ideala /W\ideala \cong \Ded /W\Ded $.
    The generators of $\ideala /W\ideala $ correspond to the elements of $(\Ded /W\Ded )^*$.
    It follows that:
    \begin{equation}
        \mgn{\br{ \alpha \in \ideala /W\ideala \mid \alpha \te{ generates }\ideala /W\ideala } }
        = \totient (W).
    \end{equation}
    By Lemma \ref{lem:trimming-small-alpha}, we see that the set $A \cap \paren{\ideala _{\le M} \setminus \ideala (\Nrm (\ideala ) R ) }$
    decomposes into the sum of $\totient (W)$ disjoint subsets according to the mod $W$ classes.
    By the pigeonhole principle, it follows that for some residue class $[b]\in \ideala /W\ideala $ we have
    \begin{equation}\label{eq:pigeonhole}
        \mgn{ \br{\alpha \in A\cap
            \paren{
            \ideala _{\le M}\setminus\ideala (\Nrm (\ideala ) R )
            }
        \mid \alpha = [b] \te{ in }\ideala /W\ideala } }
        \ge
        \frac 1{\totient (W)}\cdot (\te{R.H.S of}\eqref{eq:L/N(a)-bis})
    .\end{equation}
    Choose one such $[b]\in\ideala/W\ideala $.
    Let us fix a $C>0$ depending only on $\ideala $ and $W$ such that the projection $\ideala _{\le C} \to \ideala /W\ideala $ is surjective (which exists because the target is a finite set)
    and choose a lift $b\in \ideala _{\le C}$ of $[b]$.
%
    Let $\Aff _{W,b}\colon \ideala \to \ideala $ be the affine linear map $\alpha \mapsto W\alpha +b$.
    Set
    \begin{equation}
        B:= \Aff _{W,b}\inv (A\setminus\ideala (\Nrm (\ideala ) R) )\subset \ideala  .
    \end{equation}
    We have the following inclusion if $\param >1$ is large enough:
    \begin{equation}\label{eq:inclusion-Aff-Aff}
        \Aff _{W,b}(\ideala )\cap (\ideala _{\le M}) \subset \Aff _{W,b}(\ideala _{\le N})
        .
    \end{equation}
    Indeed, suppose $\alpha \in \ideala $ satisfies $\lnorm{W\alpha +b} \le { M} = { N\lnorm W }$.
    Since $\lnorm b \le C < N\lnorm W $ for $\param $ large enough, by the ultrametricity of $\lnorm -$ this implies $\lnorm{W\alpha }\le N\lnorm W$.
    We get $\lnorm \alpha \le N$ because $W$ is a multiplicative element for the norm $\lnorm -$.
    This proves the inclusion \eqref{eq:inclusion-Aff-Aff}.
    In particular the set on the left hand side of \eqref{eq:pigeonhole} is contained in $\Aff _{W,b} (B\cap \ideala _{\le N})$.

    Having fixed $W$ and $b$, we can finally define a function $\lambda \colon \ideala \to \bbR _{\ge 0}$ by the formula:
    \begin{equation}\label{eq:def-of-lambda}
        \lambda (\alpha ) := \frac 1{\log R} \frac{\totient (W)}{\lnorm W ^n} \frac{\kappa _\Ded }{C_\chi } \Manga (W\alpha +b )^2.
    \end{equation}
    By \eqref{eq:conseq-GY} the function $\lambda $ is $(R^{2m}/q,\rho ,S,\ded )$-pseudorandom.
    Let us verify the other hypotheses in the \Sz\ theorem \ref{thm:Sz}.

    By Lemma \ref{lem:clear-from-definitions}, the restriction of $\lambda $ to $B$ equals the constant function $\frac{\totient (W)}{\lnorm W ^n} \frac{\kappa _\Ded }{C_\chi } \log R $.
    This together with
    \eqref{eq:L/N(a)-bis}
    and
    \eqref{eq:inclusion-Aff-Aff}
    implies:
    \begin{align}
        \bbE \Bigpa{
            \lambda\ichi _B  \Bigm|
            \ideala _{\le N}
        }
        &\ge \left. \frac{1}{\totient (W)}\delta
        C_{\Ded }
        \frac{L}{\log _q L }
        \cdot \frac{\totient (W)}{\lnorm W ^n}\frac{\kappa _\Ded }{C_\chi }\log R
        \right/
        \mgn{ \ideala _{\le N} }.
    \end{align}
We have $\mgn{ \ideala  _{\le N} }\le M^n / \lnorm W ^n \Nrm (\ideala )q^{g-1}$ (which is an equality if $N$ happens to be a power of $q^d$).
By the definition \eqref{eq:def-of-our-parameters} of our parameters
we get for all sufficiently large $\param > 1$:
\begin{equation}\label{eq:def-of-delta1}
    \ge \frac 1 2 \delta C_{\Ded }
    \frac{\kappa _\Ded }{C_\chi }
    \frac{c_\domain q^{g-1}\log (q) }{n(2m+1)}
    =: \delta _1.
\end{equation}
This establishes one of the two requirements in the relative \Sz\ theorem \ref{thm:Sz}.

We have to establish one more inequality to invoke the relative \Sz\ theorem. 
By Lemma \ref{lem:clear-from-definitions} we have
$\lambda ^k\ichi _B  \le \mathrm{const.}\cdot (\log R)^{2k} $
where the constant comes from the coefficient in the definition of $\lambda $ in \eqref{eq:def-of-lambda}. So:
\begin{equation}\label{eq:smallness}
    \bbE \Bigpa{ \lambda ^k\ichi _B  \Bigm| \Ded _{\le N} }
    \le \mathrm{const.}\cdot \paren{\frac{1}{2m+1}\log N }^{2k} ,
\end{equation}
which is $< \gamma N $ for $\param $ 
sufficiently large.

Now fix $\param $ so that it satisfies \eqref{eq:for-infinitely-many-d} and
is large enough to make all the above inequalities true. 
We can apply the relative \Sz\ theorem \ref{thm:Sz} to the current situation
by \eqref{eq:def-of-delta1}, \eqref{eq:smallness} and the pseudorandomness of $\lambda $.
It follows that $B$ contains an $\ded $-homothetic copy of $S$.
Sending it by the affine $\ded $-linear map $\Aff _{W,b}\colon B \to A $, we get an $\ded $-homothetic copy of $S$ in $A$.
This completes the proof of Theorem \ref{thm:homothetic}.
\end{proof}

\begin{remark}
    The above proof actually shows a finitary version of the theorem as in \cite[Theorem A]{KMMSY} because the dependence of the threshold for $\param $ on the set $A$ is via its density $\delta $ 
    (though of course the specific value of $\param $ should be determined depending on $A$ to ensure \eqref{eq:for-infinitely-many-d}).
\end{remark}

\begin{remark}\label{rem:positive-density-assumption}
    The assumption that $A$ has positive upper density in $\pele $ was used solely at
    \eqref{eq:for-infinitely-many-d}.
    It follows that we could have assumed 
    more directly
    that $A\subset \pele $ satisfies an inequality of the form:
    \begin{equation}\label{eq:axiomatic-approach}
        \mgn{A\cap
            \ideala _{\le M}
            } > \mathrm{const.}\cdot \frac{M^n}{\log M}
    \end{equation}
    for arbitrarily large $M$, with the positive constant depending only on $\ideala $ and $A$.
    See \cite[\S\S 8-9]{KMMSY} for a fully axiomatic treatment in the number field context.
    In fact, not surprisingly, this inequality is equivalent to $A$ having positive upper density in $\pele $; see \cite[Proposition 8.14]{KMMSY} for the arguments in the number field case, which is also valid here.
\end{remark}

\subsection{Non-normal case}\label{sec:non-normal-case}
It is routine to deduce Theorem \ref{thm:intro} from Theorem \ref{thm:density}.
Let $\Or $ be an integral domain finitely generated over $\Fq $ and of transcendence degree $1$.
Let $\Ded $ be its normalization.
By Theorem \ref{thm:homothetic} and Remark \ref{rem:positive-density-assumption} it suffices to show:
\begin{proposition}\label{prop:positive-density}
    The following inequality holds for infinitely many $M\in \bbN $:
    \begin{equation}\label{eq:enough-primes-of-Or}
        \mgn{\mcal P_{\Or } \cap
        \Ded _{\le M}
        }
        >
        \mathrm{const.}\cdot \frac{M^n}{\log M}
    \end{equation}
    with the positive constant depending only on $\Or $.
\end{proposition}
We give only sketches. See also \cite[\S 10]{KMMSY} for a detailed account in the setting of number fields.
Let $\cond $ be the conductor:
\begin{equation}
    \cond := \br{ \alpha \in \Ded \mid \alpha \Ded \subset \Or } ,
\end{equation}
which is an ideal of $\Ded $ contained in $\Or $.
Let $\mathcal P_{\Ded }^{\cond }$ be a temporary notation for the set of prime elements of $\Ded $ coprime to $\cond $, which is $\mcal P_{\Ded }$ minus finitely many associate classes.

One shows that the elements of $\mcal P_{\Or }$ are precisely those elements of $\mcal P_{\Ded }^{\cond }$ which are contained in $\Or $.
More explicitly, we have the following cartesian diagram:
\begin{equation}\label{eq:cartesian}
    \xymatrix{
        \mathcal P_{\Or } \ar|{\subset }@{}[r]\ar[d]& \mathcal P _{\Ded }^{\cond } \ar[d]^{p}
        \\
        (\Or / \cond )^*\ar|{\subset }@{}[r]& (\Ded / \cond )^*
    }
.\end{equation}

Set $\Ded ^*_{\cond }:= \{ f\in \Ded ^* \mid f \mod \cond =1 \te{ in }\Ded /\cond  \} $.
It is a subgroup of $\Or ^*$ which is of finite index in $\Ded ^*$.
Let $\domain \subset \Ded \nonzero $ be a norm-length compatible $\Ded ^*_{\cond }$-fundamental domain
whose existence easily follows from Proposition \ref{prop:norm-length-compatible-domain}.
%
It suffices to show the inequality \eqref{eq:enough-primes-of-Or} with
$\mcal P_{\Or }\cap \domain $ in place of $\mcal P_{\Or }$.  
By the norm-length compatibility of $\domain $, 
we are reduced to showing the following inequality for infinitely many $L\in \bbN $:
\begin{equation}
    \mgn{
        \mcal P_{\Or }\cap \domain \cap \Ded (L)
    }
    > \mathrm{const.}\frac{L}{\log _q L}
,\end{equation}
with the constant depending only on $\Or $ and $\domain $.
Hence it suffices to show the $\alpha _0 =1$ case (say) of the following claim:
\begin{claim}\label{claim:positive-density-alpha-0}
    For each $\alpha _0\in (\Ded /\cond )^*$,
    write $\mcal P_{\Ded ,\alpha _0}^{\cond }:=p\inv (\alpha _0)$
    where $p$ is the vertical map
    in diagram \eqref{eq:cartesian}.
    Then the following inequality holds for all sufficiently large $L> 1$:
    \begin{equation}
        \mgn{
            \mcal P_{\Ded ,\alpha _0}^{\cond }\cap \domain \cap \Ded (L)
        }
        > \mathrm{const.}\frac{L}{\log _q L}
    ,\end{equation}
    with the positive constant depending only on $\cond $ and $\Ded $.
\end{claim}

Recall the definition of the Picard group $\Pic (\Ded ,\cond )$ with modulus $\cond $ as a quotient of a free abelian group:
\begin{equation}
    \Pic (\Ded ,\cond )=\frac{\bbZ [ |\spDed |\setminus \Spec (\Ded /\cond )]}{\br{(\alpha )\mid \alpha \in \Fq (\spDed )^*_{\cond }  } }
,\end{equation}
where $\Fq (\spDed )^*_{\cond }$ is the subgroup of $\Fq (\spDed )^*$ consisting of $\alpha $ which are regular around $\Spec (\Ded /\cond )$ and are equal to $1$ in $\Ded /\cond $.
Given an element $\alpha _0\in (\Ded /\cond )^*$, consider a lift $\wti \alpha _0 \in \Ded $ and its divisor $(\wti \alpha _0)$.
Its class in $\Pic (\Ded ,\cond )$ does not depend on the choice of $\wti \alpha _0$ so we get a well-defined class $[(\alpha _0)]\in \Pic (\Ded ,\cond )$.
Consider the following set of prime ideals:
\begin{align}
    \Spec (\Ded )_{\alpha _0}
    :=
    &\br{ \idealp \in \Spec (\Ded ) \mid \idealp = \wti\alpha _0\Ded \te{ for some lift }\wti \alpha _0 \in \Ded \te{ of }\alpha _0 }
    \\
    =&\br{ \idealp \in \Spec (\Ded ) \mid [\idealp ]= [(\alpha _0)] \te{ in }\Pic (\Ded ,\cond ) }
.\end{align}
The obvious map $\mcal P_{\Ded ,\alpha _0}^{\cond }\cap \domain \to \Spec (\Ded )_{\alpha _0}$; $\alpha \mapsto \alpha \Ded $ is a bijection.

The \Cheb\ Density Theorem \ref{thm:Cheb} holds with $\Pic (\kurve )$ replaced by $\Pic (\kurve ,\cond )$
with the same proof because the result \cite[Theorem 9.24, p.141]{Rosen} we cited is stated in this generality.
Thus for every finite quotient $G$ of $\Pic (\kurve ,\cond )$, an element $P\in G$ and $n>0$, we have:
\begin{multline}
        \mgn{ \Bigl\{ x\in \kurve \setminus \Spec (\Ded /\cond ) \mid \deg (x)=n\te{ and }[x]=P \te{ in }G \Bigr\} }
        \\
        =
        \begin{cases}
            \displaystyle \frac{D}{\mgn G} \frac{q^n}n + O_{\kurve }\paren{\frac{q^{n/2}}n }
            & \te{ if } n\equiv \deg (P) \te{ in }\bbZ / D\bbZ ,
            \\[5mm]
            \displaystyle  O_{\kurve} \paren{\frac{q^{n/2}}n } & \te{ else. }
        \end{cases}
    \end{multline}
We apply this to $G=\Pic (\Ded ,\cond )$ and its element $[(\alpha _0)]$. 
It follows for $n\equiv \deg (\alpha _0) $ in $\bbZ /D\bbZ $, we have
\begin{equation}\label{eq:growth-Ded-alpha-0}
    \mgn{ \br{\alpha \in \mcal P_{\Ded ,\alpha _0}^{\cond }\cap \domain \ \middle|\ \Nrm (\alpha )=q^n} }
    = \frac{D}{\mgn{ \Pic (\Ded ,\cond ) } } \frac{q^n}n + O_{\kurve} \paren{\frac{q^{n/2}}n }
.\end{equation}
This 
proves Claim \ref{claim:positive-density-alpha-0}
and hence Proposition \ref{prop:positive-density}.


\subsection*{Acknowledgements}
I have learned much of the technique used here through collaboration \cite{KMMSY} with Masato Mimura, Akihiro Munemasa, Shin-ichiro Seki and Kiyoto Yoshino.
Especially I owe much to Shin-ichiro, who was crazy enough to give us a 100-hour lecture series and teach us everything about the classical Green-Tao theorem. 
I thank Federico Binda for motivating conversations over lunch.
Most of this work was done in the latter half of 2020.
Amid all the irregularities caused by the COVID-19 pandemic, the Tohoku University staff has been so great
that I was able to finish this work more quickly than I intended.
During the work I was partially supported by JSPS KAKENHI Grant Number JP18K13382.

\end{document}